\documentclass{amsart}
\usepackage{amscd}
\usepackage{amssymb}
\input xy
\xyoption{all}
\newcommand{\comment}[1]{}  %to comment out chunks of text
\newcommand{\mathdot}{{\mathbf{\scriptscriptstyle\bullet}}}

\def\Ncyz{\Ncy_{\circ}}
\def\Ab{{\bf Ab}}
\def\ProAb{\operatorname{pro-\Ab}}

\def\twomaps{\rightrightarrows}
\def\oDelta{\overline{\delta}}
\def\BHM{B\"okstedt-Hsiang-Madsen }
\def\Ord{\mathbf{Ord}}
\def\S{{{\mathbb S}^1}}

\def\ip#1{\left< #1 \right>}
\def\ad#1{{\scriptstyle \left[ \frac{1}{#1} \right]}}

\def\hocolim{\operatorname{hocolim}}
\def\holim{\operatorname{holim}}
\def\hofi{\operatorname{hofiber}}

\def\Ncy{N^{cy}}
\def\tNcy{\widetilde{N}^{cy}}
\def\smsh{\wedge}
\def\Smsh{\bigwedge}
\def\cdh{{\text{cdh}}}
\def\scdh{{\text{scdh}}}
\def\zar{{\text{zar}}}
\def\red{{\text{red}}}
\def\nil{\operatorname{nil}}

\def\chr{\operatorname{char}}

\def\H{{\mathbb H}}

\def\tH{\widetilde{H}}

\def\ker{\operatorname{ker}}

\def\cA{\mathcal A}

\def\cF{\mathcal F}
\def\cG{\mathcal G}
\def\cI{\mathcal I}

\def\cO{\mathcal O}
\def\cMpc{{\mathcal M}_{\text{pc}}}
\def\cMpctf{{\mathcal M}_{\text{pctf}}}

\def\nor{\mathrm{nor}}
\def\sn{\mathrm{sn}}
\def\cK{\mathcal K}
\def\cKH{\mathcal KH}
\def\Hom{\operatorname{Hom}}

\def\End{\operatorname{End}}
\def\Maps{\operatorname{Maps}}
\def\Spec{\operatorname{Spec}}
\def\MSpec{\operatorname{MSpec}}
\def\bu{\mathdot}
\def\map#1{\, {\buildrel #1 \over \lra}\, }
\def\smap#1{\, {\buildrel #1 \over \to}\, }

\def\lra{\longrightarrow}
\def\into{\rightarrowtail}
\def\onto{\twoheadrightarrow}

\def\tOmega{\widetilde{\Omega}}
\def\tDelta{\widetilde{\Delta}}

\DeclareMathOperator*{\colim}{colim}

\newcommand{\Q}{\mathbb{Q}}

\newcommand{\A}{\mathbb{A}}
\newcommand{\bH}{{\mathbb{H}}}

\newcommand{\R}{{\mathbb{R}}}
\newcommand{\F}{\mathbb{F}}

\newcommand{\Z}{\mathbb{Z}}
\newcommand{\N}{\mathbb{N}}

\newcommand{\fc}{{\mathfrak c}}
\newcommand{\fp}{{\mathfrak p}}

\numberwithin{equation}{section}

\theoremstyle{plain} %% This is the default, anyway
\newtheorem{thm}[equation]{Theorem}
\newtheorem{cor}[equation]{Corollary}
\newtheorem{lem}[equation]{Lemma}
\newtheorem{prop}[equation]{Proposition}
\newtheorem{substuff}{Remark}[equation]
\newtheorem{EWT}[equation]{Equivariant Whitehead Theorem}

\theoremstyle{definition}
\newtheorem{defn}[equation]{Definition}

\newtheorem{ex}[equation]{Example}

\theoremstyle{remark}
\newtheorem{rem}[equation]{Remark}
\newtheorem{subrem}[substuff]{Remark} %for numbering x.y.1 etc
\newtheorem{subex}[substuff]{\bf Example} %for numbering x.y.1 etc

\newtheorem{construction}[equation]{Construction}

\begin{document}

%\bibliographystyle{plain}

%%% In the title, use a double backslash "\\" to show a linebreak:
%%% Use one of the following two forms:
%%% \title{Text of the title}
%%% or
%%% \title[Short form for the running head]{Text of the title}

\title[$K$-theory of toric varieties in positive characteristic]
{The $K$-theory of toric varieties \\ in positive characteristic}
\date{\today}

\author{G. Corti\~nas}
\thanks{Corti\~nas' research was supported by Conicet
and partially supported by grants UBACyT W386, PIP
112-200801-00900, and MTM2007-64704 (Feder funds).}
\address{Dept.\ Matem\'atica-Inst.\ Santal\'o, FCEyN,
Universidad de Buenos Aires,
Ciudad Universitaria, (1428) Buenos Aires, Argentina}
\email{gcorti@dm.uba.ar}

\author{C. Haesemeyer}
\thanks{Haesemeyer's research was partially supported by NSF grant DMS-0966821}
\address{Dept.\ of Mathematics, University of California, Los Angeles CA
90095, USA}
\email{chh@math.ucla.edu}

\author{Mark E. Walker}
\thanks{Walker's research was partially supported by NSF grant DMS-0966600.}
\address{Dept.\ of Mathematics, University of Nebraska - Lincoln,
  Lincoln, NE 68588, USA}
\email{mwalker5@math.unl.edu}

\author{C. Weibel}
\thanks{Weibel's research was supported by NSA and NSF grants.}
%NSA grant FA9550-0810053 and NSF grant DMS-0801060.}
\address{Dept.\ of Mathematics, Rutgers University, New Brunswick,
NJ 08901, USA} \email{weibel@math.rutgers.edu}

\begin{abstract}
%If $X$ is a toric variety over a field of finite characteristic,
%we show that the direct limit of the $K$-groups of $X$
We show that if $X$ is a toric scheme over a regular ring containing a field of finite
characteristic then the direct limit of the $K$-groups of $X$ taken
over any infinite sequence of nontrivial dilations is homotopy
invariant. This theorem was known in characteristic~0. The affine
case of our result was conjectured by Gubeladze.
\end{abstract}

\subjclass[2000]{14F20,19E08, 19D55}

\keywords{Toric variety, algebraic $K$-theory, monoid scheme,
topological Hochschild homology}

\maketitle

\tableofcontents

\section*{Introduction}\label{introduction}

Let $X$ be a toric variety over a field $k$. For each integer $c\ge2$,
multiplication by $c$ on the lattice associated to $X$ induces an
endomorphism $\theta_c$ on $X$, called a {\it dilation}. For example,
if $k = \Z/p$, the endomorphism $\theta_p$ coincides with the
Frobenius map. Locally, $\theta_c$ is determined by its action on
affine space: $\theta_c(a_1,a_2,\dots)=(a_1^c,a_2^c,\dots)$.
Each sequence $(c_1,c_2,\dots)$ of integers $\ge2$ yields a sequence
of dilations on $X$ and hence a sequence of endomorphisms $\theta_{c_i}$
of its $K$-theory $K_*(X)$ and its homotopy $K$-theory $KH_*(X)$.
The {\it Dilation theorem} for $K$-theory states that
the resulting `dilated' $K$-theory and $KH$-theory agree on toric varieties:
\begin{equation}\label{intro-Dilation}
\varinjlim\nolimits_{\,\theta_{c}} K_*(X) \map{\cong}
\varinjlim\nolimits_{\,\theta_{c}} KH_*(X).
\end{equation}
When $k$ has characteristic~0, this was proven by Gubeladze
for affine toric varieties in \cite{Gu05}, and
in full generality by the authors in \cite{chww-toric}.

If $R$ is a commutative regular $k$-algebra, we may consider
\eqref{intro-Dilation} with $X$ replaced by $X_R=X\times_k\Spec(R)$.
The main result of this paper is that the Dilation theorem
\eqref{intro-Dilation} holds for $X_R$,
even at the level of spectra. This is proven in characteristic~0
in Theorem \ref{dilation-0} and in finite characteristic in
Theorem \ref{MainTheorem} below.

Putting the affine cases of these results together,
and taking the direct limit over finitely generated submonoids,
we obtain Theorem \ref{MainTheorem-intro} (\ref{Gconjecture} below);
it settles a conjecture of Gubeladze (\cite[1.1]{Gu05})
about monoid algebras $k[A]$ when $k$ is any regular ring.
The affine case in characteristic~0 was proven by Gubeladze in \cite{Gu08}.

\begin{thm}\label{MainTheorem-intro}
Let $A$ be a cancellative, torsionfree commutative
monoid with no non-trivial units. Then for every sequence
$\fc = (c_1,c_2,\dots)$ of integers $\ge2$, and every
regular ring $k$ containing a field,
there is an isomorphism
\[
K_*(k) \map{\cong} \varinjlim\nolimits_{\theta_{c}} K_*(k[A]).
%\cK(k[A])^\fc = \cK(k).
\]
\end{thm}

Our proof of the Dilation theorem follows the geometric approach used by the
authors in \cite{chww-toric}. That proof
used the Chern character from $K$-theory to cyclic homology to
reduce everything to a problem about the cohomology of K\"ahler
differentials, which could be solved by explicit calculation using
Danilov's sheaves $\tOmega$.

In order to adapt the strategy of that proof,
two main difficulties have to be overcome. First, resolution of
singularities is not available in positive characteristic, and the
usual combinatorial resolution of singularities for toric varieties
is not good enough to control the effect on
algebraic $K$-theory and similar invariants. Overcoming
this difficulty makes it necessary to study the geometry of monoid
schemes and their resolutions; this was done in the companion paper
\cite{chww-monoid} using the Bierstone-Milman theorem.
Second, the ``correct" Chern character to use in
positive characteristic is the cyclotomic trace; its target is
topological cyclic homology, which is much more complicated than
cyclic homology. As a result, it is not sufficient (as in characteristic~0)
to study the cohomology of K\"ahler differentials and
Danilov's sheaves $\tOmega$; instead, we need to study a homotopy
theoretical version of Danilov's sheaves of differentials. The bulk of
this paper is dedicated to that task.

From an aesthetic point of view, the proof contained in this paper and
its companion \cite{chww-monoid} has the advantage that, at its heart, it
is a theorem about monoid schemes --- algebraic data (such as the base field)
only enters the story in its very late stages.
It is an intriguing question whether the original problem
--- the Dilation theorem --- can be formulated (and proved) completely
within the world of monoid schemes (or schemes over ``the field with
one element"). Such a formulation would describe a kind of
$K$-theory functor for monoid schemes, one which is not homotopy invariant.
(All of the current candidates for such a $K$-theory are homotopy
invariant in a strong way; dilations on affine space induce
equivalences so the analogues of \eqref{intro-Dilation} become trivial.)

Here is a more detailed overview of the content of this paper.
Section \ref{sec:monoids} recalls various notions from the world of
monoid schemes as developed in the companion paper \cite{chww-monoid}.
The notions of a partially cancellative monoid (resp., monoid scheme),
seminormal monoid (resp., monoid scheme) and the process of
seminormalization are of special importance for this paper. We remark that toric monoids are torsion-free cancellative and normal, so in particular they are pctf and seminormal.
Section \ref{sec:cycbar} recalls cyclic sets and their realizations
as well as the cyclic bar construction on monoids, $\Ncy$, as used by
B\"okstedt, Hsiang and Madsen \cite{BHM} in their construction of
topological cyclic homology. Section \ref{sec:compare} studies the effect
inverting a sequence of non-trivial dilations has on the cyclic nerve
of a monoid. We introduce $\tNcy$, a variant of the cyclic bar construction
in \ref{def:tNcy} that has better technical properties with regard to
ideals of monoids but is equivalent to the usual cyclic bar
construction after inverting a sequence of dilations.

In Section \ref{sec:tOmega}, we introduce the presheaf of $\S$-spaces
$\tOmega$ that is crucial for our proof of the main theorem. This
presheaf --- whose definition is based on the variant $\tNcy$ of
the cyclic bar construction 
--- plays the role that Danilov's sheaves of differentials played
in our proof of the dilation theorem in characteristic zero
in \cite{chww-toric}.
In \ref{cofibseq}, we show that $\tOmega$ satisfies excision for ideals of
monoids. It is here that working with seminormal monoids is critical.
%and we define the presheaf on general monoid schemes by first
%composing with the seminormalization functor.
%
In Section \ref{sec:descent} we build on our work in \cite{chww-monoid}
and prove a technical result (Proposition \ref{prop:cdhdescent})
about recognizing when a presheaf of spectra satisfies $cdh$ descent
adapted to the presheaves we study in this paper.

In Section \ref{sec:dilateddescent} we apply the
results of Sections \ref{sec:tOmega} and \ref{sec:descent} to prove
that presheaves of spectra obtained from smashing a (fixed)
$\S$-spectrum with $\tOmega$ and then taking fixed points for a finite
subgroup of $\S$ satisfy $cdh$ descent. The main theorem of
Section \ref{sec:dilateddescent} (Theorem \ref{MT2p}) forms the
technical core of this article; its proof is inspired by Danilov's
proof that his sheaves satisfy $cdh$ descent ({\em cf.}
\cite{Dan1479}, although of course he does not formulate his result
this way). As a consequence we conclude in Corollary \ref{Cor624} that
the `dilated' topological cyclic homology
$\varinjlim\nolimits_{\theta_c}TC^n(-;p)$ satisfies $cdh$ descent.
%with a sequence of non-trivial dilations inverted

In Section \ref{sec:char.0}, we prove the Dilation theorem when
$k$ is a regular $\Q$-algebra.
Finally in Section \ref{sec:maintheorem}, we combine
Corollary \ref{Cor624} with the main result of \cite{chww-monoid}
to prove our main theorem, the Dilation theorem \ref{MainTheorem}
when $k$ is a regular $\F_p$-algebra and $X$ is a pctf monoid scheme. We conclude by
establishing Gubeladze's conjecture (Theorem \ref{MainTheorem-intro})
in Corollary \ref{Gconjecture}.

\goodbreak
\section{Monoids}\label{sec:monoids}

In this section, we present the basic facts about monoids and
monoid schemes that we shall need. We refer to \cite{chww-monoid}
for more details.

Unless otherwise stated, in this paper a {\em monoid} will mean
a pointed abelian monoid; i.e., an abelian monoid object in the
category of pointed sets. More explicitly, it
is a pointed set $A$ equipped with a pairing $\mu: A \smsh A \to A$
that is associative and commutative and has an identity element.

We usually write the pairing $\mu$ of $A$ as multiplication $\cdot$,
in which case the identity element is written $1$ and the basepoint as
$0$, so that $1 \cdot a = a = a \cdot 1$ and $0\cdot a =0=a\cdot0$,
for all $a \in A$. For example, if $R$ is a commutative ring, then
we have the underlying multiplicative monoid $(R,\cdot)$.  If $A$ is
an unpointed abelian monoid (i.e., an abelian monoid object in the
category of sets), we write $A_*$ for the pointed abelian monoid formed
by adjoining a basepoint.

A monoid morphism $f:A\to B$ is a function preserving multiplication,
with $f(0)=0$, $f(1)=1$. If $B$ is the finite union of subsets $Ab_i$,
$b_i\in B$, we say that $f$ is {\it finite}. For a multiplicatively closed
subset $S$ of $A$ (which may or may not contain $0$), there is a
localization morphism $A\to S^{-1}A$, where the monoid $S^{-1}A$ is
defined as the set of fractions $a/s$ 
with $a \in A$ and $s \in S$ with the usual equivalence relation.

Here is some standard terminology.
An {\it ideal} $I$ in a monoid $A$ is a pointed subset such that
$AI\subseteq I$. If $I \subset A$ is an ideal, $A/I$ is the monoid
obtained by collapsing $I$ to $0$. A proper ideal is {\it prime} if
its complement $S=A\backslash\fp$ is multiplicatively closed. In this
case, we write $S^{-1}A$ as $A_\fp$.

For a monoid $A$, we write $\MSpec(A)$ for the set of its prime ideals. We
equip $\MSpec(A)$ with the ``Zariski'' topology, whose
closed subsets are those of the form
$$
V(I) := \{ \fp \in \MSpec(A) \,| \, I \subset \fp\}
$$
where $I$ is ideal of $A$. In contrast with the analogous setting of
rings, every monoid has a unique maximal ideal, given as the complement
of the set of units. The set $\MSpec(A)$ comes equipped with a sheaf
of monoids $\cA$, whose stalk at $\fp \in\MSpec(A)$ is $A_\fp$.
As explained below, $\MSpec(A)$ is an affine ``monoid scheme.''

Given an ideal $I$ of $A$, define its radical to be
$$
\sqrt{I} := \{a \in A \,| \,
	     \text{$a^n \in I$ for some $n \geq 1$}\}.
$$
An ideal $I$ is {\it radical} if $I = \sqrt{I}$.
It is easy to prove (using Zorn's Lemma) that
$\sqrt{I}$ is the intersection of the prime ideals containing $I$.
The {\it nilradical} of a monoid $A$ is the radical of the zero ideal
\begin{align}\label{eqdef:nil}
\nil(A) :=& \sqrt{\{0\}}\\
         =&\{a \in A \, | \, a^n = 0, \text{ for some $n \geq 1$}\}\nonumber.
\end{align}
Equivalently, $\nil(A)$ is the intersection of all the prime ideals of $A$.
For general monoids, the passage from $A$ to $A/\nil(A)$
is not as useful as the corresponding notion for commutative rings,
so we introduce a slightly stronger notion.

A monoid $A$ is said to be {\em reduced} if whenever $a^2 = b^2$
and $a^3 = b^3$ for some $a,b \in A$, then $a=b$. Equivalently,
$A$ is reduced if and only if whenever $a^n = b^n$
for all $n \gg 0$, we have $a=b$.
For any $A$, write $A_\red$ for the monoid obtained by modding out
$A$ by the congruence relation $a \sim b$ if $a^n = b^n$ for all $n
\gg 0$. It is clear that the canonical surjection $A \onto A_\red$ is
universal among maps from $A$ to reduced monoids.
If $A$ is reduced, then $\nil(A)=0$ because $x^n = 0$ implies $x=0$.
The converse does not hold for all monoids, but we shall see that it
holds for pc monoids (defined below) by Proposition \ref{prop:red}.

\subsection{Cancellative and partially cancellative monoids}

A (pointed commutative) monoid $C$ is {\em cancellative}
if whenever $ac = bc$ for some $c\ne0$, we have $a=b$.
Equivalently, $C$ is cancellative if $C \setminus \{0\}$ is an
unpointed monoid that maps injectively to its group completion
$(C\setminus \{0\})^+$. In this situation, the {\em pointed group
completion} of $C$ is
$$
C^+ := \left((C \setminus \{0\})^+\right)_*.
$$
We say a monoid $A$ is {\em torsionfree} if whenever $a^n=b^n$ for $a,b\in A$
and some $n\ge1$, we have $a=b$.  If $C$ is cancellative then
$C$ is torsionfree if and only if
$C^+\setminus\{0\}$ is a torsionfree abelian group.

\begin{ex} \label{ex:red}
Every cancellative monoid $C$ is reduced.
Indeed, if $x^n = y^n$ for all
$n \gg 0$, then $y^ny = y^{n+1} = x^{n+1} = x^nx = y^nx$ for all $n
\gg 0$ and hence $x=y$.
\end{ex}

If $A$ is a cancellative monoid, the {\em normalization} of $A$,
written $A_\nor$, is the submonoid of the pointed group completion
$A^+$ of $A$ consisting of all elements $\alpha \in A^+$ such that
$\alpha^n \in A$ for some $n \geq 1$. We say $A$ is {\it normal}
if $A=A_\nor$. It is easy to see (and proven in \cite[1.6.1]{chww-monoid})
that $\MSpec(A_\nor)\to\MSpec(A)$ is a homeomorphism.
The normalization of non-cancellative monoids is not defined.

\begin{defn}\label{def:pc}
A monoid $A$ is {\em partially cancellative}, or {\em pc} for short,
if $A$ is isomorphic to $C/I$ where $C$ is a cancellative monoid
and $I$ is an ideal of $C$.

A monoid $A$ is {\em partially cancellative and torsionfree}, or
{\em pctf} for short, if $A$ is isomorphic to $C/I$ where $C$ is a
torsionfree cancellative monoid
% whose group completion is a torsionfree pointed group
and $I$ is an ideal of $C$.
\end{defn}

\begin{rem}\label{pc-inspiration}
The name {\it partially cancellative} is inspired by
the following observation:
If $ac = bc$ in a pc monoid, then either $ac=bc=0$ or $a=b$.

To see this suppose $A = C/I$ with $C$ cancellative.
Given $a,b,c \in C$, if $ac = bc$ holds in $A$, then
either $ac, bc \in I$ or $ac=bc$ in $C$ and hence $a=b$.
\end{rem}

\goodbreak
\begin{prop}\label{prop:pc}
Let $\fp$ be a prime ideal in a pc monoid $A$. Then:
\begin{enumerate}

\item 
$A/\fp$ is a cancellative monoid.

\item
If $A$ is pctf, then 
$(A/\fp)^+$ is a torsionfree pointed group.

\item If $A$ is cancellative and normal, then so is $A/\fp$.
\end{enumerate}
\end{prop}

\begin{proof}
Say $A = C/I$ with $C$ cancellative and $\fp=\fp'/I$.
Then $A\setminus\fp=C\setminus\fp'$, so assertion $(1)$ follows from
the observation that, for $a,b,c \in C$ with $c\not\in\fp'$,
if $ac,bc\in\fp'$ then $a,b\in\fp'$.
If in addition $C^+$ is torsionfree, suppose that
$a^n=b^n$ in $A/\fp$. Then either $a^n=b^n$ in $C$ or $a^n,b^n\in\fp$.
In either case we must have $a=b$ in $A/\fp$ since $\fp$ is prime
and $C$ is torsion-free. This proves  assertion $(2)$.
For (3), suppose that $a,b,c\in A\setminus\fp$ satisfy $(a/b)^n=c$ in
$(A/\fp)^+$. Then $a^n=b^nc$ in the cancellative monoid $A/\fp$, and
hence in $A$.
When $A$ is normal, this implies that $a=bx$ for some $x\in A$;
since $a=bx$ also holds in $A/\fp$, we have $a/b\in A/\fp$.
\end{proof}

\begin{prop} \label{prop:red} Assume $A$ is a pc monoid.
$A$ is reduced if and only if $\nil(A) = 0$. Moreover,
$A_\red = A/\nil(A)$.
\end{prop}

\begin{proof}
Suppose $A = C/I$ with $C$ cancellative, $\nil(A) = 0$ and $x^n = y^n$
for all $n \gg 0$. Then either $x^n = y^n = 0$ for all $n \gg 0$ or
$x^n = y^n$ holds in $C$ for all $n \gg 0$. In the former case,
$x=y=0$ since $\nil(A) = 0$; in the latter case $x=y$ holds in $C$
and hence in $A$ since $C$ is reduced by Example \ref{ex:red}.
The final assertion is an immediate consequence.
\end{proof}

\begin{subrem}
Propositions \ref{prop:pc} and \ref{prop:red} fail when the monoid $A$
is not pc, because the quotient monoid $A/\fp$ by a prime ideal $\fp$
need not be cancellative, or reduced.
\end{subrem}

\medskip
\paragraph{\bf Seminormal monoids}
The notion of a seminormal monoid plays a central role in this paper.
Many of our structural results have parallels in the
theory of seminormal rings.

\begin{defn}
A monoid $A$ is {\em seminormal} provided $A$ is reduced and whenever
$x, y \in A$ satisfy $x^3 = y^2$, there is a $z \in A$ such that $x=
z^2, y=z^3$. Since $A$ is reduced, such a $z$ is unique.
\end{defn}

\begin{ex} \label{group:sn}
If $M$ is an abelian group, then the associated pointed monoid $M_*$
is seminormal: Given $x,y \in M_*$ with $x^3 = y^2$,
  either $x=y=0$ or $x = z^2, y = z^3$ where $z = y/x$.
If $A$ is a submonoid of $M_*$ and $x,y\in A$ then $z\in A_\nor$
because $z^2$ is in $A$. Thus normal monoids are seminormal.
\end{ex}

\begin{lem} \label{lem:snloc1}
If $A$ is seminormal and $S$ is a multiplicatively closed subset,
then $S^{-1}A$ is seminormal.
\end{lem}
\goodbreak

\begin{proof} Say $x,y\in S^{-1}A$
  satisfy $x^3 = y^2$. We may assume $x = a/s$, $y=b/s$
%%\sfrac{a}{s}, y = \sfrac{b}{s}$
for $a,b \in A, s \in S$.
Then $a^3s^2u = b^2s^3u$ for some $u \in S$ and so
$$
(asu^2)^3 =a^3s^2u (s u^5) =b^2s^3u (s u^5) = (bs^2u^3)^2.
$$
Since $A$ is seminormal, there is a $c \in A$ with
$$
asu^2 = c^2, \quad bs^2u^3 = c^3
$$
and hence if we set $z=c/(su)$ we have
$$
x = a/s = c^2/(s^2u^2) = z^2,  \qquad
%\sfrac{c^2}{u^2s^2} = z^2,
y = b/s = c^3/(u^3s^3) = z^3. \qedhere
%\sfrac{b}{s} = \sfrac{c^3}{u^3s^3} = z^3.
$$
\end{proof}
%where $z = \sfrac{c}{us}$.

\begin{lem} \label{lem:snq}
Suppose $A$ is a pc seminormal monoid and $I$ is an ideal
  of $A$. The monoid $A/I$ is seminormal
if and only if $I$ is a radical ideal.
\end{lem}

\begin{proof} One implication is obvious, so
suppose $I$ is radical and $x^3=y^2$ holds in $A/I$. Either $x^3=y^2=0$
(in which case $x=y=0$ since $A/I$ is reduced by Proposition \ref{prop:red}),
or $x^3 = y^2$ holds in $A$.
In this case, since $A$ is seminormal,
the equations $x=z^2$, $y=z^3$ hold in $A$ for some $z \in A$, and hence
we obtain such equations in $A/I$ as well.
\end{proof}

\subsection{Seminormalization}
%\begin{defn}
For a monoid $A$, a {\em seminormalization} of $A$ is a
seminormal monoid $B$, together with a map of monoids $A \to B$
such that the induced map $A_\red \to B$ is injective, and
for every $b \in B$, we have $b^n \in A_\red$ for all $n \gg 0$.

Seminormalizations are universal with respect to maps from $A$ to
seminormal monoids, as we now show.

\begin{lem} \label{lem:ump}
If $i: A \to B$ is a seminormalization of $A$ and $f: A \to C$
  is any map such that $C$ is seminormal, there is a unique map $g: B
  \to C$ such that $g \circ i = f$. In particular, a
  seminormalization of a monoid is unique up to unique isomorphism.
\end{lem}

\begin{proof} By the universal mapping property of $A_\red$, we
may assume $A$ is reduced, so that $A$ may be regarded as a
submonoid of $B$. Consider all pairs $(B', h)$ with
$A\subset B' \subset B$ and $h: B' \to C$ a map such that $h|_A = f$.
This is an ordered set in the evident way and by Zorn's lemma
there is a maximal element $(B',h)$. It suffices to prove $B'=B$.

If not, we may find a $b \in B \setminus B'$ such that $b^2, b^3 \in B'$.
Let $x = h(b^2), y = h(b^3)$, so that $x^3 = h(b^6) = y^2$.
Since $C$ is seminormal, there is a $z$ with $x=z^2, y=z^3$.
Consider $B'' = B'[b]$, the smallest submonoid of $B$ containing $B'$
and $b$. Every element of $B''$ may be written (non-uniquely) either
as an element of $B'$ or as $yb$ with $y \in B'$.

Consider the set $B'b$ of all elements of $B''$ of the form $t=yb$
with $y\in B'$. We claim that for each $t$ the element $h(y)z$ of $C$
is independent of the choice of $y$. If $y_1b = y_2b$ for $y_1,y_2 \in B'$,
then for all $m\ge2$:
$$
(h(y_1)z)^m = h(y_1^mb^m) = h(y_2^mb^m) = (h(y_2)z)^m.
$$
Since $C$ is reduced, it follows that $h(y_1)z = h(y_2)z$, as claimed.
Thus $h''(t)=h(y)z$ is a well defined element of $C$.
If $t\in B'$ then $h''(t)^n=h(y^n)z^n=h(yb)^n=h(t)^n$,
whence $h''(t)=h(t)$. This shows that $h$ extends to a function
$h'':B''\to C$. A similar argument shows $h''$ is a homomorphism
of monoids, contradicting the maximality of $(B',h)$. Thus $B' = B$.

%Define $\tilde{h}: B'' \to C$ to be the extension of $h$ given
% by setting $\tilde{h}(yb) = h(y)z$.  If $y_1b = y_2b$ for $y_1,y_2 \in B'$,
% then $(h(y_1)z)^m = h(y_1^mb^m) = h(y_2^mb^m) = (h(y_2)z)^m$, for all
% $m\geq 2$. Since $C$ is reduced, it follows that $h(y_1)z = h(y_2)z$.
%Similarly, if $y_1b = y_2$, then $h(y_1)z = h(y_2)$. This shows
%the definition $\tilde{h}$ is independent of choice. A similar
%argument shows $\tilde{h}$ is a homomorphism of monoids. This
%contradicts the maximality of $(B',h)$ and hence $B' = B$.

For the uniqueness, if $g_1, g_2: B \to C$ are two such maps, then for each $b \in
B$, we have $b^n \in A$ for all $n \gg 0$ and hence $g_1(b)^n = f(b^n)
= g_2(b)^n$ for all $n \gg 0$. Since $C$ is reduced, $g_1 = g_2$.
\end{proof}

If $A$ has a seminormalization, then by Lemma \ref{lem:ump}
we are justified in calling it {\em the} seminormalization of $A$,
and we write it as $A_\sn$.
It is clear that the canonical map $A \to A_\sn$ is an isomorphism
if (and only if) $A$ is seminormal, and hence that
$A_\sn \map{\cong} (A_\sn)_\sn$ is always an isomorphism.

\begin{ex} \label{ex:snc}
If $A$ is cancellative with pointed group completion $A^+$, then
$$
A_\sn = \{f \in A^+ \, | \, \text{ $f^n \in A$ for all $n \gg 0$}\}
$$
(with the canonical inclusion $A \into A_\sn$) is the
seminormalization of $A$. This follows from Example \ref{group:sn},
which shows that $A_\nor$ is seminormal.
Thus $A \subseteq A_\sn \subseteq A_\nor$.
\end{ex}
%
%It is easy to see that $A_\sn = \{ \alpha \in A_\nor \, | \, \text{
%  $\alpha^n \in A$ for all $n \gg 0$} \}$.

\begin{lem} \label{lem:snloc2}
For any monoid $A$, if its seminormalization $A \to A_\sn$ exists,
then for any multiplicatively closed subset $S$ of $A$, the map
$S^{-1}A \to S^{-1}A_\sn$ is the seminormalization of $S^{-1}A$.
\end{lem}

\begin{proof} We have $(S^{-1}A)_\red \cong S^{-1}(A_\red)$ and the map
$S^{-1}(A_\red) \to S^{-1}A_\sn$ is injective, since localization
preserves injections. By Lemma \ref{lem:snloc1}, $S^{-1}(A_\sn)$ is
seminormal. Finally, if $x \in S^{-1}(A_\sn)$, then $x^n \in
S^{-1}A$ for all $n \gg 0$.
\end{proof}

\begin{lem} \label{lem:newsn}
Let $A$ be a $pc$ monoid.
Assume its seminormalization $A \to A_\sn$ exists, and let $I$ be any
ideal of $A$. Set $J = \sqrt{IA_\sn}$. Then the induced map
$$
A/I \to A_\sn/J
$$
is the  seminormalization of $A/I$.
\end{lem}

\begin{proof}
The map $A \to A_\sn$ factors through $A_\red$ and $A/I \to A_\sn/J$
factors through $(A/I)_\red = A/\sqrt{I}$. These factorizations
allow us to assume $A$ is reduced and $I$ is radical.

We claim that $J \cap A = I$. Clearly $I\subseteq J\cap A$.
If $a \in J \cap A$, then $a^n \in IA_\sn$ for all large $n$. For any
$b \in A_\sn$ we have $b^n \in A$ for all $n \gg 0$.  It follows that
$a^m \in I$ for some $m$ (and even for all $m \gg 0$).
Since $I$ is radical in $A$, we get $a \in I$.

Since $J \cap A = I$, the map $A/I \to A_\sn/J$ is injective.
Given $y \in A_\sn/J$, it is clear that $y^n \in A/I$ for all $n \gg
0$.  Finally, $A_\sn/J$ is seminormal by Lemma \ref{lem:snq}.
\end{proof}

\begin{prop} \label{prop:snfunctor}
If $A$ is pc, then its seminormalization $A \to A_\sn$ exists
and $A_\sn$ is also a pc monoid. If $A$ is pctf then so is $A_\sn$.

Moreover, there is a functor $(-)_\sn$ from
the category of pc monoids to the category of seminormal pc monoids,
sending a monoid to its seminormalization. The maps $A \to A_\sn$
determine a natural transformation from the identity to $(-)_\sn$.
% and the functor
%  $(-)_\sn$ is idempotent up to natural isomorphism (i.e., the
%  natural map $A_\sn \map{\cong} (A_\sn)_\sn$ is an isomorphism).
\end{prop}

\begin{proof}
By assumption $A = C/I$ for a cancellative monoid $C$. By Example
\ref{ex:snc}, $C_\sn$ exists; by Lemma \ref{lem:newsn}, $A_\sn$ exists
and $A_\sn = C_\sn/J$ with $J = \sqrt{IC_\sn}$. Since $C_\sn$ is
cancellative, this shows $A_\sn$ is pc. If $C$ is also torsionfree,
then $C_\sn$ is torsionfree and hence $A_\sn$ is pctf.

Pick, once and for all, a seminormalization $A \to A_\sn$ for each pc
monoid $A$. The universal mapping property (Lemma \ref{lem:ump}) shows
that given a morphism $f: A \to B$ of pc monoids, there is a unique
map $f_\sn: A_\sn \to B_\sn$ causing the evident square to
commute. The assignments $A \mapsto A_\sn$ and $f \mapsto f_\sn$
determine a functor and the maps $A \to A_\sn$ form a natural
transformation as claimed.
\end{proof}

%\begin{subrem}\label{subrem:pctfsn}
%If $A$ is pctf, then $A_{sn}$ is pctf too. This follows from the construction
%of $A_{sn}$ in the proof of the proposition above.
%\end{subrem}

\begin{lem}\label{sn:homeo}
For any pc monoid $A$, the map $\MSpec(A_\sn) \to \MSpec(A)$
of affine monoid schemes is a homeomorphism on underlying
topological spaces.
\end{lem}

\begin{proof}
Since $\MSpec(A)\cong\MSpec(A_\red)$, we may assume $A$ reduced.
The function $\MSpec(A) \to\MSpec(A_\sn)$ sending $\fp$ to
$\tilde{\fp} := \{x \in A_\sn \, | \, x^n \in \fp \text{ for } n \gg0\}$
is a continuous inverse of the canonical map $\MSpec(A_\sn) \to\MSpec(A)$.
\end{proof}

Recall from \cite[after Remark 2.9.1]{chww-monoid} that finitely generated
monoids are {\it noetherian}: they satisfy the ascending chain condition
on ideals, and every ideal is finitely generated.

\begin{lem}\label{Asn-finite}
If $A$ is a finitely generated pc monoid, then
$A\to A_\sn$ is a finite morphism.
\end{lem}

\begin{proof}
Suppose first that $A$ is cancellative; by \cite[6.3]{chww-monoid},
$A\to A_\nor$ is finite, so $A_\nor$ is given as a finite union $\cup  Ac_i$.
Since $A$ is finitely generated, so is every ideal of $A$, including
$J_i=\{ a\in A: ac_i\in A_\sn\}$.
Writing $J_i$ as a finite union $J_i = \cup Ab_{ij}$ we have
$x_{ij} := b_{ij}c_i\in A_\sn$ and $A_\sn=\cup Ax_{ij}$.
This proves that $A\to A_\sn$ is finite. (Cf.\ \cite[2.31]{ChuLS}.)

In the general case, write $A$ as $C/I$ with $C$ cancellative; replacing
$C$ by a submonoid, we may assume that $C$ is finitely generated
(see\ \cite[Proof of part (1) of 9.1]{chww-monoid}).
By Lemma \ref{lem:newsn},
$A_\sn=C_\sn/J$ for an ideal $J$. Since $C \to C_\sn$
is finite, $A \to A_\sn$ is also finite.
\end{proof}

The {\em conductor ideal} for an inclusion $A\subset B$ is
the ideal $I$ of $A$ consisting of those $a \in A$ with
$B \cdot a \subset A$. Observe that $I$ is also an ideal of $B$.
If $A\subset B$ is finite and $S\subset A$ is multiplicatively closed,
the conductor of $S^{-1}A\subset S^{-1}B$ is $S^{-1}I$.

\begin{ex} \label{ex:sn}
Let $A$ denote the submonoid of the free monoid $\langle x,y\rangle$
generated by $x^2,xy,y^2,x^2y,xy^2$. Then $A$ is seminormal, and
$A_\nor=\langle x,y\rangle$. The conductor ideal for $A\subset A_\nor$
is $I=\{ x^my^n: mn>0\}$.
\end{ex}

\begin{lem} \label{lem:cond}
Suppose $A$ is seminormal and cancellative.
Then the conductor ideal $I$ for the inclusion $A \into A_\nor$
is a radical ideal of both $A$ and $A_\nor$.
\end{lem}

\begin{proof}
It suffices to prove $I$ is radical in $A_\nor$.
Say $\alpha \in A_\nor$ and $\alpha^n \in I$. Since $I$ is an
ideal of $A_\nor$, we have $\alpha^m \in I$ for all $m \geq n$
and hence, since $A$ is seminormal, $\alpha \in A$. For each
$\beta \in A_\nor$ and all $m\ge n$ we have
$$
(\beta \cdot \alpha)^m = (\beta^m \alpha^{m-n}) \alpha^n \in A,
$$
since $\alpha^n \in I$. Since $A$ is seminormal, we have
$\beta \cdot \alpha \in A$, proving that $\alpha \in I$.
\end{proof}

\begin{subrem}\label{conductorsheaf}
If $A$ is a finitely generated and cancellative monoid,
$A\subset A_\nor$ is finite; see \cite[6.3]{chww-monoid}.
If $I$ is the conductor and $S$ is multiplicatively closed in $A$,
it follows that $S^{-1}I$ is the conductor of $S^{-1}A\subset S^{-1}A_\nor$.
Therefore the conductor defines a sheaf of ideals on $\MSpec(A)$.
\end{subrem}

\medskip
\paragraph{\bf Seminormal monoid schemes}
By a {\it monoid scheme} we mean a topological space $X$, equipped with
a sheaf of monoids $\cA$, which can be covered by open subspaces
isomorphic to
$\MSpec(A)$ for some monoid $A$.
We refer to \cite{chww-monoid} for a more precise definition.

We now review the most relevant definitions. If $\cI$ is a (quasi-coherent)
sheaf of ideals of $\cA$, there is a monoid subscheme $Z$ whose
topological space is $V(\cI)$ and whose sheaf of monoids is $\cA/\cI$.
We call such a subscheme an {\it equivariant closed} subscheme.
A monoid scheme $X$ is of {\em finite type} if the underlying
topological space is finite and every stalk $\cA_x$ is a finitely
generated monoid. Since finitely generated monoids are noetherian,
monoid schemes of finite type have
the descending chain condition on equivariant closed subschemes.

For technical reasons, we will need our monoid schemes to be
{\it separated} in the sense of \cite[3.3]{chww-monoid}:  $X$ is separated if
the diagonal $X\to X\times X$ is a closed immersion
(see \cite[2.5]{chww-monoid}).
This notion is parallel to its counterpart in algebraic geometry.

\begin{defn}\label{def:Mpctf}
A monoid scheme $X$ is {\em pc} (respectively, {\em pctf})
if for each $x$ in $X$ the stalk monoid $\cA_x$ is pc (respectively, pctf)
in the sense of Definition \ref{def:pc}.
If $X=\MSpec(A)$ then $X$ is pc if and only if $A$ is pc
by the proof of \cite[9.1]{chww-monoid}.
We write $\cMpc$ for the category consisting of separated
pc monoid schemes of finite type, and $\cMpctf$ for
the full subcategory of separated pctf monoid schemes of finite type.
\end{defn}

\begin{subex}\label{toric+pctf}
Recall from \cite[4.1]{chww-monoid} that a {\it toric monoid scheme} is a
separated, connected, torsionfree, normal monoid scheme of finite type.
Thus toric monoid schemes are in $\cMpctf$.

In fact, a toric monoid scheme is the same thing as a
connected normal scheme in $\cMpctf$, since the stalks $\cA_x$
of a normal pctf scheme are torsionfree by Proposition \ref{prop:pc}(2).
We showed in \cite[4.4]{chww-monoid} that (up to isomorphism) any
toric monoid scheme $X$ uniquely determines a fan $(N,\Delta)$
and that $X$ may be constructed from this fan.
\end{subex}

In this paper we shall also be interested in {\it seminormal} monoid schemes,
by which we mean monoid schemes whose stalk monoids $\cA_x$ are
seminormal. Lemma \ref{lem:snloc1} implies that an affine scheme
$X=\MSpec(A)$ is seminormal if and only if $A$ is seminormal.
($A$ is the stalk at the unique closed point of $X$.)
Lemma \ref{sn:homeo} implies that $\cA_\sn$ is a sheaf of monoids
on $X=\MSpec(A)$, and that $\MSpec(A_\sn)$ is isomorphic to
the monoid scheme $(\MSpec(A),\cA_\sn)$. It follows that for any
pc monoid scheme $(X,\cA)$ there is a sheaf of monoids $\cA_\sn$
and we define the seminormalization of $X$ to be $(X,\cA_\sn)$.
The sheaf map $\cA\to\cA_\sn$ induces a natural map $X_\sn \to X$.

\begin{prop}
If $X =(X,A)$ is a pc monoid scheme,
%and $\cA_\sn$ is the functor sending $x \in X$ to $(\cA_x)_\sn$,
then $X_\sn := (X,\cA_\sn)$ is also a pc monoid scheme, and
every map from a seminormal monoid scheme to $X$
factors uniquely through $X_\sn\to X$.
\end{prop}

\begin{proof}
Since being pc and the universal mapping property are local in $X$,
we may assume that $X$ is affine. In this case, the assertions
follow from Proposition \ref{prop:snfunctor} and Lemma \ref{lem:ump}.
\end{proof}

Let $p: Y \to X$ be a morphism between monoid schemes of finite type.
We say that $p$ is {\it finite} if $X$ can be covered by affine open subschemes
$U$ such that $p^{-1}(U)\to U$ is isomorphic to $\MSpec(B)\to\MSpec(A)$
for some finite monoid map $A\to B$. (See \cite[6.2]{chww-monoid}.)
We say that $p$ is {\em birational} if there exists an open, dense
subset $U$ of $X$ such that $p^{-1}(U)$ is dense in $Y$ and the induced map
$p^{-1}(U) \to U$ is an isomorphism. (See \cite[10.1]{chww-monoid}.)

\begin{subrem}
If the conductor contains a nonzerodivisor $s\in B$, then
$A\subset B$ is birational (as $ A[1/s]=B[1/s]$).
Conversely, if $\nil(B)=0$ and $A\subset B$ is finite birational
then the conductor ideal contains a nonzerodivisor of both $A$ and $B$.
This is because if $\MSpec(B[1/s])$ is dense in $\MSpec(B)$ then
$s$ is a nonzerodivisor, and
if $B=\cup Ab_i$ and $b_i=a_i/s$ then $Bs\subset A$.
\end{subrem}

A monoid scheme is {\it reduced} if its stalks are reduced
monoids (see \cite[Sec.\,2]{chww-monoid}).

\begin{lem} \label{lem:snfinite}
If $X$ belongs to $\cMpctf$ (resp. $\cMpc$) then so does $X_\sn$
and the map $X_\sn \to X$ is a finite morphism.
If in addition $X$ is reduced, the map $X_\sn \to X$ is birational.
\end{lem}

\begin{proof}
The first assertion is immediate from Lemma \ref{Asn-finite}
and Proposition \ref{prop:snfunctor}. %Remark \ref{subrem:pctfsn}.

Assume now that $X$ is reduced.
The map $X_\sn \to X$ is a homeomorphism by Lemma \ref{prop:red}.
By \cite[10.1]{chww-monoid} it suffices to show this map induces an
isomorphism on stalks at generic points. Since $X$ is reduced, the
stalk $A$ at a generic point is the pointed monoid associated to an
abelian group, and the map on generic stalks induced by
$X_\sn \to X$ has the form $A \to A_\sn$, which is
an isomorphism as $A$ is seminormal.
\end{proof}

\begin{lem} \label{lem:snscheme}
Let $X$ be a pc monoid scheme with seminormalization $X_\sn \to X$, and
let $Z \subset X$ an equivariant
closed subscheme of $X$. Set $Y := Z\times_X X_\sn$.
Then $Y_\red \to Z$ is the seminormalization of $Z$.
\end{lem}

\begin{proof}
This is a local question so we may assume that $X=\MSpec(A)$,
$Z=\MSpec(A/I)$ and $Y=\MSpec(A_\sn/J,)$, $J=IA_\sn$. But
$Z_\sn=\MSpec(A_\sn/\sqrt{J})$ by Lemma \ref{lem:newsn}. The result follows.
\end{proof}

\section{The Cyclic Bar Construction}\label{sec:cycbar}

\subsection{The cyclic bar construction for monoids}
Let $A$ be a (pointed abelian) monoid. A two-sided pointed $A$-set is
a pointed set $B$ equipped with two pairings $A\smsh B\to B$ and
$B\smsh A\to B$ such that $1b = b = b1$, $a_1(a_2b) = (a_1a_2)b$,
$(ba_1)a_2 = b(a_1a_2)$ and $(a_1b)a_2 = a_1(ba_2)$ hold for all
$a_1,a_2 \in A$, $b \in B$. The 2-sided bar construction
$\Ncy(B,A)$ is the pointed simplicial set whose set of $n$-simplices is
\[
\Ncy(B,A)_n=B\smsh \overbrace{A\smsh\cdots\smsh A}^{n\text{ factors} }
\]
and whose face and degeneracy maps are
just like those used to define the Hochschild complex of a ring.

The multiplication action of $A$ on itself makes it a 2-sided pointed
$A$-set, and we write $\Ncy(A)$ instead of $\Ncy(A,A)$.
The operations
$$
t_n(a_0,\dots, a_n) = (a_1, \dots, a_n, a_0),
$$
make $\Ncy(A,A)$ into a pointed cyclic set (see \cite[9.6.2]{WH} and
the next section for more details).
Following \BHM \cite{BHM}, we call $\Ncy(A)$ the {\it cyclic nerve} of $A$.

For a pointed set $X$ and a ring $R$, we let $R[X]$ denote the free
$R$-module on $X$, modulo the summand indexed by the base point of $X$.
If $A$ is a pointed monoid, the free $R$-module $R[A]$ is a ring in
the usual way, with multiplication given by the product rule for $A$.
For any commutative ring $R$, the chain complex associated to the
simplicial free $R$-module $R[\Ncy(B,A)]$ is the usual Hochschild complex
relative to the base ring $R$ of the monoid-ring $R[A]$ with coefficients in
the $R[A]$-bimodule $R[B]$.  In particular, the homology of
$|\Ncy(B,A)|$ with $R$ coefficients is the Hochschild homology
relative to $R$ of $R[A]$ with coefficients in the $R[A]$-bimodule $R[B]$:
\begin{equation}\label{eq:H(N)=HH}
H_q(|\Ncy(B,A)|, R) = HH_q(R[A]/R, R[B]).
\end{equation}
\goodbreak

\begin{defn}\label{def:Ncy(A,a)}
Let $B$ be a two-sided pointed $A$-set for which the two
actions coincide (i.e., $ab = ba$ for all $a \in A, b \in B$).
For each  $b\in B$, we define $\Ncy(B,A;b)$ to be
the pointed simplicial subset of $\Ncy(B,A)$ consisting of those
$n$-simplices $(b_0,a_1, \dots, a_n)$ satisfying $b_0\cdot\prod_i a_i=b$,
together with the base-point $(*,0,\dots,0)$.
Thus $\Ncy(B,A)$ decomposes as a wedge of pointed simplicial sets
$$
\Ncy(B,A) = \bigvee_{b\in B} \Ncy(B,A;b).
$$
In particular, when $B=A$ we write $\Ncy(A,a) = \Ncy(A,A;a)$
and have $\Ncy(A)=\bigvee_{a \in A} \Ncy(A,a)$.
Each subset $\Ncy(A,a)$ is invariant under the operators $t_n$ on $\Ncy(A)$,
so each $\Ncy(A,a)$ is a cyclic subset of $\Ncy(A)$.
For example, if $A$ has no zerodivisors then each
$$
\Ncy(A,0)_n = \{(a_0, \dots, a_n)\} \, | \, \prod_{i=0}^n a_i = 0\}
$$
is a point for all $n$, i.e., $\Ncy(A,0)$ is a point,
regarded as a constant simplicial set.
\end{defn}

\begin{subrem}\label{Ncy-modI}
If $I$ is an ideal and $a\not\in I$, then $\Ncy(A,a)\smap{\cong}\Ncy(A/I,a)$.
This is immediate from the observation that if $(x_0,\dots,x_n)$ is
an $n$-simplex of $\Ncy(A,a)$ and it is not the base point, then
$x_0\cdots x_n = a\notin I$ and hence none of $x_0,\dots,x_n$ can be in $I$.
\end{subrem}

Observe that if $A = \smsh_{i \in I} A^i$ is a smash product of pointed
monoids, then
\begin{equation} \label{E310}
\Ncy(A) \cong \Smsh_{i\in I} \Ncy(A^i).
\end{equation}
In particular if $\{B^i\}_{i\in I}$ is a family of unpointed monoids, then
$\Smsh(B^i_*)=(\prod B^i)_*$ and
\begin{equation}\label{E311}
\Ncy(\prod_i B^i)_*=\Smsh_{i\in I}\Ncy(B^i_*)=(\prod_{i\in I}\Ncy B^i_*)
\end{equation}

\subsection{Cyclic sets and $\S$-spaces}
We briefly review the basic theory of $\S$-spaces and cyclic sets.
We refer the reader to \cite{BHM}, \cite[7.1]{Loday} and \cite[9.6]{WH}
for additional information.

Let $\S$ denote the Lie group of complex numbers of norm $1$.
We will often identify $\S$ with the group $\R/\Z$
via the homeomorphism induced by sending $\theta \in \R$ to
$e^{2\pi\theta i}$ in $\S$.
The only closed  subgroups of $\S$ are $\S$ itself
and the cyclic group $C_r$ of $r$-th roots of unity for each $r \geq 1$.

An {\em $\S$-space} is a topological space equipped with a continuous
action of $\S$.  An equivariant map of $\S$-spaces $f: X \to Y$ is an
{\em $\S$-weak equivalence} if $f: X^H \to Y^H$ is a weak equivalence
in the usual sense for every closed subgroup $H$ of $\S$ (i.e. for $H
= \S$ and $H = C_r, r\ge1$).  Let $I$ denote the unit interval with
trivial $\S$-action.  Two equivariant maps $f_0,f_1 : X \to Y$ are
{\em $\S$-homotopic} if there is an equivariant map
$h: X \times I \to Y$ with $f_i = h|_{X \times \{i\}}$ for $i=0,1$.
A map $f: X \to Y$ is an {\em $\S$-homotopy equivalence} if it admits
an inverse up to $\S$-homotopy.

All of the above definitions have evident pointed variants.

A  pointed $\S$-CW complex is a pointed $\S$-space  built from pointed
$\S$-cells of the form $(\S/H)_+ \smsh D^n$ with $H$ as above.
The following analogue
of the usual Whitehead theorem is useful; see \cite[I.1]{LMS}.

\begin{EWT}\label{EWT}
A morphism between $\S$-CW complexes is an $\S$-homotopy equivalence
if and only if it is an $\S$-weak equivalence.
\end{EWT}

Let $\Lambda$ denote the {\em cyclic category} with objects
$[0], [1], [2], \dots$
and whose morphisms are generated by the face and degeneracy maps as in the
simplicial category along with an automorphism $\tau_n: [n] \to [n]$
of order $n+1$ for each $n \geq 0$. See, e.g., \cite[6.1.1]{Loday}
or \cite[9.6.1]{WH} for a complete list of relations.
A {\em  cyclic set} is a contravariant functor from $\Lambda$ to the
category of sets. Since $\Lambda$ contains the simplicial category as
a subcategory, there is a forgetful functor from cyclic sets to
simplicial sets.

If $X_\mathdot$ is a cyclic set, then the geometric realization of
its underlying simplicial set is an $\S$-space (see \cite{DHK});
in fact,
$|X_\mathdot|$ is an $\S$-CW complex by \cite[Theorem 1]{FiedorowiczGajda}.

The representable functor $\Hom_\Lambda(-,[n])$ on
the cyclic category $\Lambda$ is a cyclic set, and hence a simplicial
set,  and there is a homeomorphism
\begin{equation} \label{|Lambda|}
|\Hom_{\Lambda}(-,[n])| \map{\cong} \Lambda^n := \S \times \Delta^n,
\end{equation}
where $\Delta^n := \{(u_0, \dots, u_n) \, | \, u_i \in \R, u_i \geq 0,
\sum_i u_i = 1\}$ is the standard topological $n$-simplex.
(See \cite[1.6]{BHM}.)
The collection of spaces $\Lambda^n, n\ge0$ is a cocyclic $\S$-space ---
i.e., a covariant functor from $\Lambda$ to the category of $\S$-space.
We choose the homeomorphism \eqref{|Lambda|} as in \cite[7.2]{HM},
so that for the automorphism $\tau_n: [n] \to [n]$, we have
\begin{equation} \label{craction}
(\tau_n)_*(z; u_0, \dots, u_n) =(z e^{\frac{2\pi i}{n+1}};
u_1, \dots, u_n, u_0).
\end{equation}
Note that $\Lambda^n$ is an $\S \times C_{[n]}$-space, where
$C_{[n]} = C_{n+1}$ is the group of cyclic permutations of $[n]$.
Explicitly, $\S$ acts (only) on the first component by translation,
and the cyclic group acts (only) in the second component
by permuting the vertices.

For a general cyclic set $X_\mathdot$,
we use the $\S$-spaces $\Lambda^n$ to construct the coend
$$
\int_{n \in \Lambda} X_n \times \Lambda^n
  := \coprod_n X_n \times \Lambda^n/\sim
$$
where $\sim$ is the equivalence relation generated by $(f^*x, t) \sim
(x, f_*t)$ for $f$ a morphism in $\Lambda$.
As proven in \cite[Theorem 1]{FiedorowiczGajda}, this coend is
canonically homeomorphic to the geometric realization
of the simplicial
set underlying $X_\mathdot$. We are thus justified in defining the geometric
realization of a cyclic set $X_\mathdot$ as
$$
|X_\mathdot| := \int_{n \in \Lambda} X_n \times \Lambda^n.
$$
Thus, $|X_\mathdot|$ is an $\S$-space with the action of $\S$
determined by the action on the $\Lambda^n$ of \eqref{|Lambda|}.

\begin{ex}\label{faithful-action}
For any cyclic set $X_\mathdot$, the subspace $|X_\mathdot|^\S\subset |X_\mathdot|$
is the discrete set of points given by the equalizer of the two maps
$s_0, t_1 \circ s_0: X_0 \rightrightarrows X_1$
(see \cite[page ~145]{Dunn}).  If $A$ is a monoid, then
$s_0, t_1 \circ s_0: A =\Ncy(A)_0 \to \Ncy(A)_1 = A \smsh A$ are
given by $a \mapsto a \smsh 1$ and $a \mapsto 1 \smsh a$,
and thus we have $|\Ncy(A)|^\S = \{0,1\} = S^0$ for any $A \ne 0$,
with $|\Ncy(A,a)|^\S=\{0\}$ for $a\ne1$.
In particular, the map between the cyclic nerves of non-zero monoids
induced by a monoid homomorphism is an $\S$-homotopy equivalence
if and only if it is a weak equivalence on $C_r$-fixed points
for each $r\ge 1$.
\end{ex}

\begin{ex} \label{Ex310}
Suppose $M$ is an abelian group, written additively
for convenience, and let $\Ncyz(M)$ denote the unpointed version of
the cyclic bar construction, so that an $n$-simplex of $\Ncyz(M)$ i
s an $m+1$-tuple $(m_0, \dots, m_n)$ of elements of $M$.
Note that $\Ncy(M_*) = \Ncyz(M)_*$ and we have an evident decomposition
$$
\Ncyz(M) = \coprod_{m \in M} \Ncyz(M,m).
$$

For each $m \in M$, the simplicial set $\Ncyz(M,m)$ is isomorphic
to $BM$, the standard simplicial classifying space of $M$.
The isomorphism sends $(m_0,\dots, m_n)$ to $(m_1,\dots, m_n)$ and its
inverse sends $(m_1,\dots, m_n)$ to $(m-m_1-\cdots-m_n,m_1,\dots,m_n)$.
For each $m\in M$, this isomorphism induces a cyclic structure on $BM$ given by
$$
\tau_n(m_1, \dots, m_n) = (m_2, \dots, m_n, m-m_1 - \cdots -m_n).
$$
This cyclic structure on $BM$ depends on $m$ of course, and we will describe the fixed point sets of the geometric
realization for finite subgroups of $\S$ in what follows.

There is a natural map from $|\Ncyz(M)|$ to the free loop space
$\Lambda|BM|$, adjoint to the composition of the $\S$-action
$\S\times|\Ncyz(M)|\to|\Ncyz(M)|$ with the geometric realization of the
isomorphism $\Ncyz(M)\to BM$ described above.
It is proved in \cite[Prop. 2.6]{BHM} that this map induces a homotopy
equivalence $|\Ncyz(M)|^{C_r} \to (\Lambda|BM|)^{C_r}$ for all $r\geq 1$.

Since $M$ is abelian, there is a homotopy equivalence
$\Lambda|BM|\to|BM|\times M$; the projection onto $|BM|$ is given by
evaluation at $1$, and the projection to $M=\pi_0(\Omega |BM|)$ is obtained by
deforming a loop in $|BM|$ to a loop based at the base point
and then taking its homotopy class (this is well-defined since $M$ is
abelian). It is easy to check that this map is actually an
$\S$-homotopy equivalence for the action on $|BM|\times M$ that is, on
the component $|BM|\times\{m\}$, induced by the composite map
$\S\times |BM| \overset\cong\to |B(\Z\times M)|\to |BM|$, where
$B(\Z\times M)\to BM$ is given by applying $B$ to the homomorphism
$(n,x)\mapsto n\cdot m + x$.

Under the composite map $|\Ncyz(M)| \to \Lambda |BM| \to |BM|\times M$, the component $|\Ncyz(M,m)|$ is mapped to the component $|BM|\times\{m\}$ and it follows from \cite[Proposition 2.6.]{BHM} that the induced map $|\Ncyz(M,m)|^{C_r} \to (|BM|\times\{m\})^{C_r}$ (with action as described above) is a homotopy equivalence for all $r\geq 1$.

Now consider the case where $M$ is a free abelian group of rank $n$
and let $e\in M$. It follows from the preceding discussion that
$|\Ncy(M_*,e)|^{C_r}$ is homotopy equivalent to an $n$-dimensional torus
with disjoint basepoint if $e$ is divisible by $r$ (i.e.,
if $e=r\cdot e'$ for some $e'\in M$), and to a one-point space otherwise.
\end{ex}

\begin{subex}\label{Ncy+units}
More generally, for any (pointed) monoid $A$ and any unit $a \in A$, we have
$$
\Ncy(A,a) = \Ncy(U(A)_*,a) \cong B(U(A))_*,
$$
where $U(A)$ is the group of units in $A$.
This follows from Remark \ref{Ncy-modI}.
\end{subex}

For each $r>0$, the canonical group homomorphism
$$
\rho_r:\S \onto \S/C_r \cong \S, \, z \mapsto z^r,
$$
allows us to view an $\S$-space $X$ as an $\S$-space $\rho_r^*X$ on
which $z$ acts as $z^r$. In particular, $C_r$ acts trivially on $\rho_r^*X$.

\begin{subex}\label{Ncy(x)}
Consider the free monoid $\langle x\rangle$ on one generator.
By \ref{def:Ncy(A,a)}, $\Ncy(\langle x\rangle)$ decomposes as
$\vee_{r\ge0} \Ncy(\langle x\rangle,x^r)$. By inspection, we have
$\Ncy(\langle x\rangle,1)=S^0$ and $\Ncy(\langle x\rangle,x)=S^1_*$.
For $r\ge1$, we have an $\S$-equivariant homeomorphism
$$
|\Ncy(\langle x\rangle,x^r)| \map{\sim} (\S\times_{C_r}\Delta^{r-1})_*,
$$
and the canonical map
$|\Ncy(\langle x\rangle,x^r)|\to|\Ncy(\langle x^{\pm1}\rangle,x^r)|$
is an $\S$-homotopy equivalence (see \cite[3.20--21]{Rognes} for example).
That is, $|\Ncy(\langle x\rangle,x^r)| \map{\sim}(\S/C_r)_*=\rho_r^*\S$.
 In particular, we have an $\S$-equivariant deformation retraction
\[
|\Ncy(\langle x\rangle)|\map{\sim}
S^0\vee \bigvee\nolimits_{r\ge1}(\S/C_r)_*.
\]
By convention, we set $C_0=\S$, so that
$|\Ncy(\langle x\rangle)|\map{\sim}\bigvee_{r\ge0}(\S/C_r)_*$.
\end{subex}

We will need to understand the $\S$-homotopy type of the cyclic nerve
of the free abelian pointed monoid $F_r := \langle x_1,\dots,x_r\rangle$.
Using \eqref{E310} and \eqref{E311} it suffices to consider the case $r=1$,
which is handled in Example \ref{Ncy(x)}. We record the result:

\begin{lem} \label{Ncy(free)}
For integers $e_1,\dots,e_r \ge0$, there is an $\S$-homotopy equivalence
$$
(\S/C_{e_1} \times \cdots \times\S/C_{e_r})_*
\map{\sim}
|\Ncy(F_r, x_1^{e_1} \cdots x_r^{e_r})|.
$$
Moreover, the map
$$
|\Ncy(F_r, x_1^{e_1} \cdots x_r^{e_r})|
\into
|\Ncy(\ip{x_1^{\pm1}, \dots, x_r^{\pm1}}, x_1^{e_1} \cdots x_r^{e_r})|
$$
is an $\S$-homotopy equivalence.
\end{lem}

The space $(\S/C_{e_1} \times \cdots \S/C_{e_r})_*$
in Lemma \ref{Ncy(free)} is a pointed torus,
whose dimension is the number of nonzero $e_i$.
If $e_i>0$ for all $i$, it is the space
$(\S \times \cdots \times \S)_*$ equipped with the $\S$-action given as
$$
z \cdot (w_1, \dots, w_r) = (z^{e_1}w_1, \dots, z^{e_r}w_r).
$$

\begin{ex}
Let $A_*$ be the pointed monoid associated to the integer lattice points
in the cone in $\R^2$ spanned by the vectors $v=(1,0)$ and $w=(1,3)$,
and set $x= (1,1)$, $y=(1,2)$; these are irreducible elements in $A$.
The element $a=(2,3)$ satisfies $a = x + y = v + w$.
Let $B_*$ be the pointed monoid generated by $x$ and $y$ and $C_*$ the
pointed monoid generated by $v$ and $w$. Then each of $B_*$ and $C_*$
is the pointed monoid associated to a free abelian monoid with two
generators. Since $a$ can be written as a non-trivial sum in only the
two ways given above,
we have
$$
\Ncy(A_*,a) = \Ncy(B_*,a) \bigvee \Ncy(C_*,a)
$$
and
$$
\Ncy(B_*,a) \cap \Ncy(C_*,a) = \Ncy(D_*, a),
$$
where $D_*$ is the monoid generated freely by $a$. We see that
$|\Ncy(A_*,a)|$ is the pushout in the category of $\S$-spaces of
$$
\xymatrix{
\S_* \ar[r]^{\text{diag}} \ar[d]_{\text{diag}} & (\S \times \S)_* \\
(\S \times \S)_*
}
$$
in which each map is the diagonal map $z \mapsto (z,z)$  and, for
each copy of $\S \times \S$, the group $\S$ acts diagonally:
$z(t,u) = (zt,zu)$.
\end{ex}

\subsection{Edgewise subdivision}
We recall Segal's $r$-fold edgewise subdivision $sd_r(X_\bu)$ of
a simplicial or cyclic set $X_\bu$, as presented in \cite{BHM}.
If we think of a simplicial set as a contravariant functor from the
category $\Ord$ of finite, totally
ordered, non-empty sets to the category of sets, then $sd_r(X_\bu)$ is
the functor obtained from $X_\bu: \Ord \to \mathbf{Sets}$ by
pre-composing with the endofunctor $(-)^{\amalg r}\!:\Ord \to \Ord$
sending $A$ to  $A^{\amalg r} := A \amalg \cdots\amalg A$,
with $r$ copies of $A$, indexed by $\{1,.\dots,r\}$.
%ordered from ``left to right''.
More formally, $A^{\amalg r}$ is the set $A \times \{1,\dots, r\}$,
with $(a,i) \leq (b,j)$ if $i<j$ or $i=j$ and $a\le b$.
One may also view $(-)^{\amalg r}$ as an endofunctor on the
simplicial category, which is the full skeletal subcategory of $\Ord$
consisting of objects $[0], [1],\dots$. Then
$[n]^{\amalg r} = [r(n+1) -1]$ and $sd_r(X_\mathdot)_n = X_{r(n+1)-1}$.

Recall from \cite[1.1]{BHM} that for any simplicial set $X_\mathdot$
there is a homeomorphism
\begin{equation} \label{E32b}
D_r: |sd_r(X_\mathdot)| \map{\cong} |X_\mathdot|
\end{equation}
induced by the maps
$$
(sd_r(X_\bu))_n \times \Delta^n =
X_{r(n+1)-1}\times \Delta^{n} \to X_{r(n+1)-1}\times \Delta^{r(n+1)-1}
$$
sending $(x,\vec{u})$, with $x$ in $sd_r(X_\mathdot)_{n} = X_{r(n+1)-1}$
and $\vec{u} =(u_0, \dots, u_{n}) \in \Delta^{n}$ for $n \geq 0$,
to $(x, \frac{\vec{u}}{r} \oplus \cdots \oplus \frac{\vec{u}}{r})$, with
$r$ copies of $\frac{\vec{u}}{r}$.

For $r \geq 1$, define $\Lambda_r$ to be the category with the same
objects as the cyclic category $\Lambda$ and whose morphisms have the
same generators, except that the relation on $\tau_n: [n] \to [n]$
is $\tau_n^{r(n+1)} = \text{id}$. Note that $\Lambda_1 = \Lambda$.
A $\Lambda_r$-set is a contravariant functor from $\Lambda_r$ to the
category of sets. As in the case $r=1$, $\Lambda_r$ contains the
simplicial category as a sub-category and hence a $\Lambda_r$-set has
a geometric realization.
As observed in \cite[(1.4)]{BHM}, the cyclic group $C_r$ acts on
$\Lambda_r$, with the generator acting on $[n]$ by $\tau_n^{n+1}$.
Thus $C_r$ acts naturally on any $\Lambda_r$-set.

The geometric realization of the representable $\Lambda_r$-set
$\Hom_{\Lambda_r}(-, [n])$ is homeomorphic to $\R/r\Z \times\Delta^n$,
as shown in \cite[1.6]{BHM}.  Just as in the case $r=1$, we endow the
geometric realization of a $\Lambda_r$-set $X_\bu$ with an action of the
topological group $\R/r\Z$ by using the homeomorphism \cite[1.8]{BHM}
$$
|X_\bu| \cong \int_n |\Hom_{\Lambda_r}(-,[n])|  \times X_n.
$$
Henceforth, we identify
$\R/r\Z$ with $\S$ via $z
\mapsto e^{\frac{2 z \pi i}{r}}$ so that the geometric realization of
a $\Lambda_r$-set is an $\S$-space. The following fact will be used:

\begin{lem} \label{lem:cr}
For any $\Lambda_r$-set $X_\bu$, the action of the subgroup
$C_r$ of $\S$ on $|X_\bu|$ is induced by the natural action of $C_r$
on the simplicial set $X_\bu$, where the generator acts on $X_n$ by
$\tau_n^{n+1}$. 
\end{lem}

\begin{proof}
This is established on p.\,469 of \cite{BHM} (after Definition 1.5).
\end{proof}

\begin{cor}\label{freeCr}
If $a$ is not an $r^{th}$ power in $A$ then $\Ncy(A,a)^{C_r}=\{a\}$
and $|\Ncy(A,a)|^{C_r}$ is a 1-point space.
\end{cor}

If $X_\bu$ is a cyclic set, $sd_r(X_\mathdot)$ is a $\Lambda_r$-set
on which $\tau_n \in\Hom_{\Lambda_r}([n], [n])$ acts
on $sd_r(X_\mathdot)_n = X_q$ ($q={r(n+1)-1}$)  via the action of
$\tau_{q} \in\End_{\Lambda}([q])$ on $X_{q}$.
With our convention for viewing $|sd_r(X_\bu)|$ as an $\S$-space, the
homeomorphism $D_r: |sd_r X_\bu| \map{\cong} |X_\bu|$ is $\S$-equivariant
\cite[1.11]{BHM}.

In the special case $X_\mathdot=\Ncy(A)$,
if $tA^r$ denotes $A^r$ with the twisted $A^r$-action
$$
(a_1,\dots,a_r)\cdot(x_1,\dots,x_r)\cdot(b_1,\dots,b_r) =
(a_rx_1b_1,a_1x_2b_2,\dots,a_{r-1}x_rb_r),
$$
then $sd_r(X_\mathdot)$ is again a cyclic bar construction $\Ncy(tA^r,A^r)$;
see \cite[2.1]{BHM}. It is useful to think of an $n$-simplex of
$sd_r(\Ncy(A))$ as an $r \times (n+1)$ matrix of entries in $A$.
By Lemma \ref{lem:cr} the action of $C_r$ on $|sd_r(\Ncy(A))|$ is
induced by the simplicial action that permutes the rows of such matrices.

In order to study the effect of the dilation $\theta_r$, we will need a
comparison between $\Ncy(A,a)$ and $\Ncy(A,a^r)$,
given in \eqref{eq:arCr} below;
our presentation is taken from the discussion in \cite{BHM}.
There is a natural diagonal map of simplicial sets,
$$
\delta_r: \Ncy(A) \to sd_r(\Ncy(A)),
$$
which sends $(a_0, \dots, a_n)$ to
the $r \times (n+1)$ matrix with this row repeated $r$ times.
The map of associated $\S$-spaces
$$
|\delta_r|: |\Ncy(A)| \to |sd_r(\Ncy(A))|
$$
is not $\S$-equivariant, but instead satisfies \cite[(2.7)]{BHM}:
$$
|\delta_r|(z^r \cdot p) = z\cdot |\delta_r|(p), \,
\text{ for all $p \in |\Ncy(A)|, z\in \S$.}
$$
Taking $z$ in $C_r$, this shows that
the image of $|\delta_r|$ is contained in the
$C_r$-fixed points of $|sd_r(\Ncy(A))|$.
For all $r\ge1$, the map $\delta_r$ determines an isomorphism of
simplicial sets,
$$
\delta_r: \Ncy(A) \map{\cong} sd_r(\Ncy(A))^{C_r}.
$$
As shown in \cite[(2.3)]{BHM}, this determines a homeomorphism
$$
|\delta_r|:|\Ncy(A)| \map{\cong} |sd_r(\Ncy(A))^{C_r}|=|sd_r(\Ncy(A))|^{C_r}.
$$
Using the homeomorphism $D_r$ of \eqref{E32b},
define the endomorphism $\oDelta_r$ of $|\Ncy(A)|$ as
\begin{equation} \label{E32}
\oDelta_r = D_r \circ \delta_r:
|\Ncy(A)| \map{\delta_r} |sd_r \Ncy(A)| \map{D_r} |\Ncy(A)|.
\end{equation}
By what we have shown above, we have
$$
\oDelta_r(z^r \cdot x) = z\cdot \oDelta_r(x), \,
\text{for all $x \in |\Ncy(A)|$ and $z \in \S$,}
$$
and
$\oDelta$ induces a homeomorphism
$$
\oDelta_r: |\Ncy(A)|\map{\cong} |\Ncy(A)|^{C_r}.
$$
In other words, $\oDelta_r$ is an $\S$-equivariant map
$$
\oDelta_r: \rho_r^* |\Ncy(A)| \to |\Ncy(A)|
$$
that induces an $\S/C_r$-equivariant homeomorphism of $C_r$ fixed
point subspaces:
$$
\oDelta_r: \rho_r^* |\Ncy(A)| \map{\cong}
|\Ncy(A)|^{C_r}
$$

The map $\oDelta_r$ sends  the component indexed by $a \in A$ to that
indexed by $a^r \in A$, and so
restricts to a homeomorphism
\begin{equation*}
\oDelta_r:
\bigvee_{a \in A, a^r = b} |\Ncy(A, a)| \map{\cong} |\Ncy(A, b)|^{C_r},
\end{equation*}
for each $b \in A$.
In particular, $|\Ncy(A,b)|^{C_r} = \{*\}$ if $b \in A$ is
not the $r$-th power of any element of $A$.
If $A$ is the quotient of a cancellative and torsionfree monoid by a
radical ideal, then we have a homeomorphism for each $a \in A$:
\begin{equation} \label{eq:arCr}
\oDelta_r:
|\Ncy(A, a)| \map{\cong} |\Ncy(A, a^r)|^{C_r}.
\end{equation}

\begin{ex}
Consider the free commutative pointed monoid $F_n$ on generators $x_1,
\dots, x_n$. From Lemma \ref{Ncy(free)}, we have the equivalence
for each $(e_1,\dots,e_n)$:
$$
(\S/C_{e_1} \times \cdots \times S/C_{e_n})_* \map{\sim}
|\Ncy(F_n, x_1^{e_1} \cdots x_n^{e_n})|,
$$
where $\S/C_0$ is defined to be the trivial group. For any $r\ge1$,
the action of
$C_r$ on $\S/C_{e_i}$ is fixed point free unless $r \mid e_i$, in which
case  $C_r$ acts trivially. Thus the action of $C_r$ on
$|\Ncy(F_n, x_1^{e_1} \cdots x_n^{e_n})|$ fixes only the basepoint
unless $r \mid e_i$ for all $i$.
If $r \mid e_i$ for all $i$, the action of $C_r$ is trivial, and
the homeomorphism
$$
|\Ncy(F_n, x_1^{\frac{e_1}{r}}, \dots, x_n^{\frac{e_n}{r}})|
\map{\cong}
|\Ncy(F_n, x_1^{e_1}, \dots, x_n^{e_n})|^{C_r} =
|\Ncy(F_n, x_1^{e_1}, \dots, x_n^{e_n})|
$$
of \eqref{eq:arCr}
corresponds, under the equivalences of Lemma \ref{Ncy(free)},
to the homeomorphism
$$
(\S/C_{\frac{e_1}{r}} \times \cdots
\S/C_{\frac{e_n}{r}})_* \map{\cong}
(\S/C_{e_1} \times \cdots \times S/C_{e_n})_*
$$
whose inverse is the map that  raises elements to the $r$-th power.
\end{ex}

\goodbreak
\section{The Comparison Theorem}\label{sec:compare}

\subsection{The dilated cyclic bar construction}
We now introduce a variant on $\Ncy$.
If $a$ is an element of a monoid $A$, we write $A\ad{a}$ for
the monoid formed by adjoining an inverse of $a$; formally,
$A\ad{a}=S^{-1}A$ where $S=\{a^n \,|\, n \ge0\}$.

\begin{defn}\label{def:tNcy}
For a monoid $A$, define
$$
\tNcy(A) = \bigvee_{a \in A} \Ncy(A\ad{a}, a),
$$
where the element $a$ in $\Ncy(A\ad{a}, a)$ refers to the element $a/1$
of $A\ad{a}$.
Thus an $n$-simplex of $\tNcy(A)$
is given by $(a; \alpha_0, \dots, \alpha_n)$ where
$\alpha_0, \dots, \alpha_n \in A \ad{a}$, $a \in A$,
and the equation $\alpha_0 \cdots \alpha_n=a/1$ holds in $A\ad{a}$.
\end{defn}

As we saw in Definition \ref{def:Ncy(A,a)},
$\Ncy(A\ad{a},a)$ is a cyclic set for each $a$, so $\tNcy(A)$ is a
cyclic set. For each $a\in A$, $\Ncy(A,a)\to\Ncy(A\ad{a},a)$ is a
morphism of cyclic sets, and hence there is a morphism of cyclic sets
$$
\Ncy(A) \to \tNcy(A)
$$
that is natural for homomorphisms of monoids.
\comment{ %%%%%%%%%%%%%%%%%%%%%%%%%%%%%%%%%%%%%%%%%%
Since $a$ is a unit of $A\ad{a}$, Example \ref{Ncy+units} shows that
$$
\Ncy(A\ad{a},a) \cong BU(A\ad{a})_*
$$
as simplicial sets, where $U(-)$ denotes the %sub
group of units of a monoid and $B$ refers
to the classical bar construction.
Thus $|\tNcy(A)|$ is (non-equivariantly) homotopy equivalent to
$$
\bigvee_{a \in A} |B(U(A\ad{a}))|_*,
$$
a wedge sum of pointed, abelian $K(\pi,1)$ spaces.

For the examples we care about most in this paper, $A$ and $a$ will have the
property that $L=U(A\ad{a})$ is a free abelian group of finite
rank. By Example \ref{Ex310}, it follows that for all $r\geq 1$
$\Ncy(A\ad{a}, a)^{C_r}$ is homotopy equivalent to
\[
\left[L_\R/L\right]^{C_r}_* \cong \left[L\otimes_\Z\S\right]^{C_r}_* .
\]
The action of $\S = \R/\Z$ on $L_\R/L \cong L\otimes_\Z\S$
is given here (using additive notation as in Example \ref{Ex310}) by
$$
\overline{z} \cdot \overline{b} = \overline{za + b} \pmod{L},
$$
for any $b \in U(\ad{a})_\R$.
} %%%%%%%%%%%%%%%% end commented out material %%%%%%%%%%%%%%

In nice enough situations, $\Ncy$ and $\tNcy$ coincide up to equivalence:

\begin{prop}
If $A = (\N^r \times G)_*$ where $G$  is an abelian group, then
the natural map $|\Ncy(A)| \to |\tNcy(A)|$ induces an $\S$-homotopy
equivalence.
\end{prop}

\begin{proof}
The result is immediate if $A = G_*$, and  follows from
Lemma \ref{Ncy(free)} if $A = \N^r_*$.
The general case follows from this using \eqref{E310} and \eqref{E311}.
\end{proof}

The monoids $A$ appearing in the proposition above with $G$ finitely
generated are those for which $\Q[A]$  is smooth. (If
the group $G$ has $p$-torsion, $k[A]$ will not be smooth for
fields $k$ of characteristic $p$.) We thus think of the failure of the
map $|\Ncy(A)| \to |\tNcy(A)|$ to be an equivalence as measuring that
lack of ``smoothness of  a monoid''.

The functor $\tNcy$ is better behaved than $\Ncy$ for
``pathological'' monoids.  An example of this behavior is
the following result.  Recall from \eqref{eqdef:nil} that
$\nil(A)$ denotes the ideal of nilpotent elements in $A$.

\begin{lem} 
The map $\tNcy(A) \to \tNcy(A/\nil(A))$ is an isomorphism.
\end{lem}

\begin{proof}
If $a \in A$ is nilpotent,
then $A\ad{a}$ is the trivial pointed monoid $\{0\}$, and
$\Ncy(A\ad{a},a)=\Ncy(0)$ is the one-point space. It follows that
$$
\tNcy(A) = \bigvee_{a \in A \setminus \nil(A)} \Ncy(A \ad{a}, a).
$$
We also have
$$
\tNcy(A/\nil(A)) =
\bigvee_{a \in A \setminus \nil(A)} \Ncy((A/\nil(A)) \ad{a}, a)
$$
and so it suffices to prove that for all $a\in A\setminus\nil(A)$,
the map
%\addtocounter{equation}{-1}
%\begin{subequations}
%\renewcommand{\theequation}{\theparentequation.\arabic{equation}}
\begin{equation*} %\label{E628}
\Ncy(A \ad{a}, a) \to
\Ncy((A/\nil(A)) \ad{a}, a)
\end{equation*}
%\end{subequations}
is an isomorphism. But this follows from Remark \ref{Ncy-modI},
since $(A/\nil(A))[1/a]=A[1/a]/\nil(A[1/a])$ and $a\notin \nil(A[1/a])$.
\end{proof}

\begin{subex}\label{tNcy(cusp)}
If $B$ is the free monoid on $x$, it follows from Example \ref{Ncy(x)}
that $|\tNcy(B)|$ is a bouquet of copies of $\S/C_r$ indexed by $r\ge1$.
If $A$ is the submonoid $\langle x^2,x^3\rangle$, then $|\tNcy(A)|$ is
the sub-bouquet indexed by $r\ge2$.
\end{subex}

\subsection{Comparison Theorem}
The Comparison Theorem \ref{CT} below concerns the map $\Ncy\to\tNcy$.
Roughly,  it says that the map becomes an equivariant homotopy
equivalence upon inverting enough dilations. Intuitively,
upon inverting dilations, every (partially cancellative) monoid
resembles a smooth monoid.

For any positive integer $c$, the dilation
$$
\theta_c: A \to A, \, a \mapsto a^c.
$$
defines a monoid endomorphism of a monoid $A$. If $c = p$ is prime,
then the induced map on monoid-rings  $\theta_p: \Z/p[A] \to \Z/p[A]$
coincides with the usual Frobenius.

Now let $\fc = (c_1, c_2, \dots)$ be an infinite list of integers with
$c_i \geq 2$ for all $i$, and define $A^\fc$ to be the abelian monoid
$$
A^\fc := \colim\{A \map{\theta_{c_1}} A \map{\theta_{c_2}} \cdots\}.
$$
We think of this construction as a form of localization, since, for
example, if $c_i = c$ for all $i$, then $(-)^\fc$ amounts to inverting
the endomorphism $\theta_c$.
Similarly, if $F$ is any functor on monoids, we define $F(A)^\fc$ by
\begin{equation}\label{F(A)c}
F(A)^\fc = \colim\{ F(A)\map{\theta_{c_1}} F(A) \map{\theta_{c_2}} \cdots\},
\end{equation}
provided this (sequential) colimit exists in the target category.
If $F$ commutes with sequential colimits, then clearly
$F(A)^\fc = F(A^\fc)$.
For example, both $\Ncy$ and $\tNcy$ commute with filtered colimits;
this is easy for $\Ncy$ and straightforward for $\tNcy$.
Hence \eqref{F(A)c} defines the cyclic sets
$\Ncy(A)^\fc=\Ncy(A^\fc)$ and $\tNcy(A)^\fc=\tNcy(A^\fc)$.

If $\rho: A \to A'$ is a homomorphism of pointed monoids, then we may
regard $A'$ as a two-sided pointed $A$-set (with the two actions
coinciding) in the obvious way. Observe that raising elements to the
power of $c$ determines an endomorphism of the simplicial set $\Ncy(A',A)$,
which we will also write as $\theta_c$. More generally, suppose $B$ is a
pointed $A$-subset of $A'$  that is closed under $b \mapsto b^c$. Then
$\theta_c$ restricts to an endomorphism of simplicial sets on $\Ncy(B,A)$.
\goodbreak

\begin{lem}\label{Ac=Asnc}
Let $A$ be a monoid and let $\fc = (c_1,c_2, \dots) $ be an
infinite sequence of integers with $c_i \geq 2$ for all $i$. Then
\begin{enumerate}
\item $A^\fc \to (A_\red)^\fc$ is an isomorphism;
\item If $A$ is pc then $A^\fc\to(A_\sn)^\fc$ is an isomorphism.
\end{enumerate}
\end{lem}

\begin{proof}
Since $\pi:A\to A_\red$ is a surjection, so is $A^\fc\to (A_\red)^\fc$.
Suppose that $a,b\in A$ agree in $(A_\red)^\fc$. Then there is a
$c=c_1\cdots c_n$ so that $\pi(a^c)=\pi(b^c)$ in $A_\red$.
Thus either $a^c=b^c$ in $A$ or else $a,b\in\nil(A)$ and
$a^n=b^n=0$ for $n\gg0$. In either case, $a$ and $b$ agree in $A^\fc$.

Now assume that $A$ is pc, so that $A_\sn$ exists.
By (1), we may assume that $A$ is reduced.
By definition, $A \to A_\sn$ is an injection; it follows that
$A^\fc \to (A_\sn)^\fc$ is an injection. To see that it is a surjection,
fix $t\in A_\sn$. If $c=c_1\cdots c_n$ and $n$ is a suitably large integer,
then $a=t^c$ is in $A$; in particular $a=\theta_c(t)$. Thus the image of $t$ in
$(A_\sn)^\fc$ is in $A^\fc$.
\end{proof}

\begin{thm} \label{CT}
Let $A$ be a pctf monoid and let $\fc = (c_1,c_2, \dots)$ be an
infinite sequence of integers with $c_i \geq 2$ for all $i$.
Then the natural maps of cyclic sets
\begin{equation*}
\Ncy(A)^\fc \to \tNcy(A)^\fc \map{\cong} \tNcy(A_\sn)^\fc
\end{equation*}
are $\S$-homotopy equivalences (on geometric realizations).
\end{thm}

The proof of the Theorem \ref{CT} occupies the remainder of this section.
We begin with a series of reductions.
The first reduction is that we may assume $A$ seminormal,
since Lemma \ref{Ac=Asnc} 
implies that $\Ncy(A^\fc)\cong\Ncy((A_\sn)^\fc)$
and  $\tNcy(A^\fc)\cong\tNcy((A_\sn)^\fc)$.

Next, observe that both cyclic sets are amalgamated sums
indexed by $A^\fc$. An element of $A^\fc$ is given by an element
$a\in A$ occurring at some stage in the colimit of
$(A \map{\theta_{c_1}} A \map{\theta_{c_2}} \cdots)$.
Dropping the first few terms from $\fc$ does not change the colimit
$A^\fc$ used to index the terms.  Thus, to prove the theorem, it
suffices to show that for all nonzero $a\in A$,
\addtocounter{equation}{-1}
\begin{subequations}
\renewcommand{\theequation}{\theparentequation.\arabic{equation}}
\begin{equation} \label{E38}
\Ncy(A,a)^\fc \to \Ncy(A\ad{a},a)^\fc = \tNcy(A,a)^\fc
\end{equation}
induces an $\S$-homotopy equivalence on geometric realizations.
In this equation, $\Ncy(A,a)^\fc$ is the colimit of
$$
\Ncy(A,a) \map{\theta_{c_1}} \Ncy(A, a^{c_1})
\map{\theta_{c_2}} \Ncy(A, a^{c_1c_2}) \to \cdots
$$
and similarly for $\tNcy(A,a)^\fc$, with $A$ replaced by $A\ad{a}$.

%We may assume $a$ is not nilpotent (i.e., that $a^n \ne 0$ for all
%$n\geq 1$). For if $a$ is nilpotent, the target
%of \eqref{E38} is a one point simplicial set, and so is the source,
%since $a^{c_1 \cdots c_n}= 0$ for $n \gg 0$.

We also claim that it suffices to assume $A$ is cancellative and torsion
free. Indeed, since $a$ is not nilpotent,
there is a prime ideal $\fp$ of $A$ such that $a \notin\fp$.
Since $A$ is pctf, the monoid $A/\fp$ is cancellative and
$(A/\fp)^+$ is torsionfree by Proposition \ref{prop:pc}.
%For an $n$-simplex $(x_0,\dots, x_n)$ of $\Ncy(A,a)$, we have
% $x_0 \cdots x_n = a$ and so $x_i\notin \fp$ for all $i$. It follows that
Since $(A/\fp)\ad{a} \cong A\ad{a}/\fp\ad{a}$, the maps
$$
\Ncy(A, a) \to \Ncy(A/\fp, a)
\quad\text{and}\quad
\Ncy(A\ad{a}, a) \to \Ncy((A/\fp)\ad{a}, a)
$$
are isomorphisms by Remark \ref{Ncy-modI}. Our claim follows.

By the equivariant Whitehead Theorem \ref{EWT},
the map \eqref{E38} is an $\S$-homotopy equivalence provided
\begin{equation} \label{E528d}
|\Ncy(A^\fc,a)|^H =
|\Ncy(A,a)^\fc|^H\to |\Ncy(A\ad{a},a)^\fc|^H = |\Ncy(A\ad{a}^\fc,a)|^H
\end{equation}
\end{subequations}
\noindent
is a (non-equivariant) weak equivalence for each closed subgroup $H$ of $\S$.
By Example \ref{faithful-action}, \eqref{E38} is a homotopy equivalence
for $H=\S$ (because $a\ne1$ in $A^\fc$ implies $a\ne1$ in $A\ad{a}^\fc$).
Thus to prove Theorem \ref{CT} it suffices to
show that \eqref{E528d} is a weak equivalence for $H =C_r$, $r\ge1$.
In fact, as we next show, we may assume $r=1$:
\goodbreak

\begin{lem}
For any cancellative, torsionfree monoid $A$, if the map
\eqref{E38} is a weak equivalence for all $a$ then
the map \eqref{E528d}
is a weak equivalence for all $a$ and all closed subgroups $H\subset\S$.
\end{lem}

\begin{proof}
Fix $H=C_r$.  Since $A$ is torsionfree, if $a = b^r$ for some $b\in A$,
then $b$ is unique and we have the natural homeomorphism \eqref{eq:arCr}
for $A$ and $A\ad{a}$.
Then the assertion that \eqref{E528d} is a homotopy equivalence for $a$
is equivalent to the assertion that the map
\addtocounter{equation}{-1}
\begin{subequations}
\renewcommand{\theequation}{\theparentequation.\arabic{equation}}
\begin{equation} \label{E528c}
|\Ncy(A,b)^\fc| \to |\Ncy(A\ad{b},b)^\fc|
\end{equation}
\end{subequations}
is a (non-equivariant) homotopy equivalence, which is the hypothesis.
More generally, if $a^{c_1 \cdots c_n} = b^r$ for some $n$
and $b \in A$, then it suffices to prove \eqref{E528c} is a homotopy
equivalence (since we can omit the first $n-1$ terms from each
colimit); again, this is a special case of the hypothesis that
\eqref{E38} is a homotopy equivalence.

Finally, suppose $a^c$ is not the $r$-th
power of an element of $A$, for any $c = c_1 \cdots c_n$ with $n \geq 1$.
Then the source of \eqref{E528d} is a one-point space, by \ref{freeCr},
and it suffices to prove that the same holds for the target.
Again by \ref{freeCr}, it suffices to prove
that $a^c$ is not the $r$-th power of an element of $A\ad{a}$ for any
such $c$. Say $a^c = u^r$ for some $u \in A\ad{a}$. Writing
$u = \frac{b}{a^m}$, for some $b \in A$ and $m\ge0$, it follows that\
$$
u^{rm+1} = (u^r)^mu = a^{mc}\frac{b}{a^m} = a^{m(c-1)}b \in A.
$$
Since $r$ and $rm+1$ are relatively prime, there is a positive integer
$L$ (specifically, $L = mr(r-1)$) such that every $l\ge L$
is a non-negative integer linear combination of $r$ and $mr+1$.
It follows that $u^l \in A$ for all $l \geq L$. But then
$a^{c_1 \cdots c_N} = (u^r)^{c_{n+1} \cdots c_N}$ in $A$ for $N\gg0$,
a contradiction to the assumption that $a^c=u^r$ in $A\ad{a}$.
\end{proof}

In summary, we have reduced the proof of Theorem \ref{CT} to proving that
\eqref{E38} is a (non-equivariant) weak  equivalence when $A$ is
cancellative, seminormal and torsionfree and $a \ne 0$.
Thus we need to prove that the map
\begin{multline*}
\colim\left\{
|\Ncy(A,a)| \map{\theta_{c_1}}
|\Ncy(A,a^{c_1})| \map{\theta_{c_2}} \cdots
\right\}
\to \\
\colim\left\{
|\Ncy(A\ad{a},a)| \map{\theta_{c_1}}
|\Ncy(A\ad{a},a^{c_1})| \map{\theta_{c_2}} \cdots
\right\}
\end{multline*}
is a weak equivalence.
As with any filtered system of simplicial sets,  the
colimits of these sequences of maps coincide with their homotopy colimits
(up to homotopy equivalence) and these homotopy colimits
are given by the mapping telescope construction; see \cite[XII.3.5]{BKan}.
It thus suffices to prove that
\begin{multline*}
\hocolim\left\{
|\Ncy(A,a)| \map{\theta_{c_1}}
|\Ncy(A,a^{c_1})| \map{\theta_{c_2}} \cdots
\right\}
\to \\
\hocolim\left\{
|\Ncy(A\ad{a},a)| \map{\theta_{c_1}}
|\Ncy(A\ad{a},a^{c_1})| \map{\theta_{c_2}} \cdots
\right\}
\end{multline*}
is a homotopy equivalence. The advantage of using homotopy colimits
(i.e., mapping telescopes) is that we may replace the $\theta_{c_i}$
with homotopy equivalent maps without affecting the homotopy colimit.

For a two-sided $A$-set $B$ for which the two actions of $A$ on $B$
coincide, if we regard an $n$-simplex of $sd_r(\Ncy(B,A))$ as an
$r\times(n+1)$ matrix then taking products of the columns
defines a simplicial map
\begin{align*}
\mu_r: sd_r(\Ncy(B,A)) &\to \Ncy(B,A), \\
\mu_r(b, a_{1,1}, \dots, a_{1,n}, a_{2,0}, \dots, a_{r,n}) &=
\left((b a_{2,0} \cdots a_{r,0}), (a_{1,1} \cdots a_{r,1}), \cdots)\right.
\end{align*}
Following \cite[2.5]{BHM}, for any simplicial set $X_\mathdot$,
define a continuous map on geometric realizations
$$
\psi_r: |sd_r(X_\bu)| \to |X_\bu|
$$
by sending $(x,\vec{u})$ in $sd_r(X)_n\times\Delta^{r(n+1)-1}$ to
$(x,\vec{0}\oplus\cdots\oplus\vec{0}\oplus\vec{u}) =
(d_0^{(r-1)(n+1)}(x), \vec{u})$ in $|X_\bu|$. Here,
$d_0^{(r-1)(n+1)}$ is the iterated simplicial face map.
The map $\psi_r$ is homotopic to the map $D_r$ defined in \eqref{E32b}
via the homotopy of \cite[2.5]{BHM}:
\[
H_t(x, \vec{u})=
(x,(t/r)\vec{u}\oplus\dots\oplus(t/r)\vec{u}\oplus(1-t+t/r)\vec{u}).
\]

Here is a picture summarizing the spaces and maps we have introduced
so far:
$$
\xymatrix{
|\Ncy(B,A)| \ar@{-->}[dr]^{\theta_r} \ar@{-->}[r]^{|\delta_r|} &
|sd_r \Ncy(B,A)| \ar[d]^{\mu_r}
\ar@<1ex>[rr]^{D_r \phantom{X} \cong} \ar@<-1ex>[rr]_{\psi_r} &&
|\Ncy(B,A)| \\
& |\Ncy(B,A)| \\
}
$$
The dotted arrows are defined when $A=B$. One may easily verify
that the triangle commutes when $A=B$.
The map $\psi_r \circ|\delta_r|$ is induced by the simplicial map
$$
\alpha_r: \Ncy(A) \to \Ncy(A)
$$
sending $(a_0, \dots, a_n)$ to
$(a_0^r  \left(\prod_{i=1}^n a_i\right)^{r-1}\!,\ a_1,\dots,a_n).$
In other words, for each $m$ in $A$, the map
$$
\alpha_r: \Ncy(A,A;m) \to \Ncy(A,A;m^r)
$$
is induced by multiplication by $m^{r-1}$ on the first copy of $A$.

Even though $\delta_r$ is not defined unless $A=B$,
the simplicial map
$$
\alpha_r: \Ncy(B,A; m) \to \Ncy(B,A; m^r)
$$
is defined for any $A$-subset $B$ of the group completion $A^+$
and any $m \in A$. In this situation,
$\alpha_r$ is the map induced by the endomorphism
$b\mapsto bm^{r-1}$ of $B$.  We define the map
$$
\theta'_r: |\Ncy(B,A; m)| \to |\Ncy(B,A; m^r)|
$$
to be the composition $\theta'_r = \mu_r \circ D_r^{-1} \circ \alpha_r$.
When $B=A$, $\theta'_r = \mu_r\circ D_r^{-1}\circ\psi_r\circ\delta_r$
is homotopic to $\theta_r$.

Since in any mapping telescope one may replace the maps by homotopic
ones without affecting the homotopy type, to prove the theorem it
suffices to show that
\begin{multline*}
\hocolim \left\{ |\Ncy(A,m)| \map{\theta'_{c_1}}
|\Ncy(A,m^{c_1})| \map{\theta'_{c_2}} \cdots\right\}
\to \\
\hocolim \left\{ |\Ncy(A\ad{m},m)| \map{\theta'_{c_1}}
|\Ncy(A\ad{m},m^{c_1})| \map{\theta'_{c_2}} \cdots\right\}
\end{multline*}
is a homotopy equivalence.

\begin{lem}\label{Ncy(Am,A)}
The natural inclusion map
\begin{equation*} 
|\Ncy(A\ad{m},A;m)| \into |\Ncy(A\ad{m},A\ad{m};m)|
\end{equation*}
is a homotopy equivalence.
\end{lem}

\begin{proof}
As observed in Example \ref{Ncy+units},
the right hand side is the classifying space of the group $U(A\ad{m})$,
since $\Ncy(A\ad{m},m) = \Ncy(U(A\ad{m}),m)$.
Similarly, if we let $U_A$ denote $A\cap U(A\ad{m})$ then the left hand side is
$$
\Ncy(A\ad{m},A;m) = \Ncy(U(A\ad{m}),U_A;m),
$$
and the map $\Ncy(U(A\ad{m}),U_A;m) \to B(U_A)$
sending $(b,a_1, \dots, a_n)$ to $(a_1, \dots, a_n)$ is an
isomorphism. This shows that the geometric realization
$|\Ncy(A\ad{m},A;m)|$ is the classifying
space of the monoid $U_A$, and the map in the lemma %\eqref{E49} i
is the map on classifying spaces induced by the inclusion of $U_A$
into $U(A\ad{m})$. Since $U(A\ad{m})$ is the group completion of
$U_A$, Lemma \ref{Ncy(Am,A)} follows from the fact that the
classifying space of an abelian monoid coincides (up to homotopy
equivalence) with that of its group completion.
\end{proof}

We have thus reduced the proof of Theorem \ref{CT} to showing that
the map from
$$
\hocolim \left\{|\Ncy(A,m)| \map{\theta'_{c_1}}
|\Ncy(A,m^{c_1})| \map{\theta'_{c_2}} \cdots\right\}
$$
to
$$
\hocolim \left\{|\Ncy(A\ad{m},A;m)| \map{\theta'_{c_1}}
|\Ncy(A\ad{m},A; m^{c_1})| \map{\theta'_{c_2}} \cdots\right\}
$$
is a homotopy equivalence for all $m \in A$.

To do this, we form a filtration of $A\ad{m}$ by $A$-subsets by
defining, for each non-negative integer $l$, the $A$-set
$$
B_l = \{b \in A\ad{m} \mid bm^l \in A\}.
$$
Observe that $B_0$ is $A$ and $\cup_{l\geq 0} B_l$ is $A\ad{m}$.
We get a filtration of spaces
$$
|\Ncy(A,m)| = |\Ncy(B_0,A;m)| \subset |\Ncy(B_1,A;m)|
\subset \cdots \subset |\Ncy(A\ad{m},A;m)|
$$
with $|\Ncy(A\ad{m},A;m)| = \bigcup_{l \geq 0} |\Ncy(B_l,A;m)|$.
Moreover, the map $\theta'_r$ preserves this
filtration, since each of the maps $\alpha_r$, $D_r$ and $\mu_r$
does. It therefore suffices to prove that
the map from
$$
\hocolim \left\{|\Ncy(B_l,A;m)| \map{\theta'_{c_1}}
|\Ncy(B_l,A; m^{c_1})| \map{\theta'_{c_2}} \cdots\right\}
$$
to
$$
\hocolim \left\{|\Ncy(B_{l+1},A;m)| \map{\theta'_{c_1}}
|\Ncy(B_{l+1},A; m^{c_1})| \map{\theta'_{c_2}} \cdots\right\}
$$
is a homotopy equivalence, for all $l \geq 0$. But this is clear:
For any $r \geq 2$, the map
$$
|\Ncy(B_{l+1},A;m)| \map{\theta'_r} |\Ncy(B_{l+1},A; m^r)|
$$
lands in the subspace $|\Ncy(B_{l},A; m^r)|$,
% \subset |\Ncy(B_{l+1},A; m^r)|$,
since $\alpha_r$ sends the $n$-simplex $(b,a_1, \dots, a_n)$ to $(bm^{r-1},a_1, \dots, a_n)$, and if $bm^{l+1} \in A$, then $bm^{r-1}m^l
\in A$.

This completes the proof of the Comparison Theorem \ref{CT}.

\bigskip\goodbreak
\subsection{Topological Hochschild homology}

We now interpret Theorem \ref{CT} in terms of topological Hochschild homology.

For a ring $k$, we write $TH(k)$ for the {\em topological Hochschild homology}
spectrum of $k$. It is a cyclotomic spectrum in the sense of
\cite[2.2]{HM}; see \cite[2.4]{HM} for details.
For a prime $p$ and integer $n\ge1$, one defines the ordinary spectrum
\begin{equation}\label{def:TR(k)}
TR^n(k; p) := TH(k)^{C_{p^{n-1}}}.
\end{equation}
There are natural maps of ordinary spectra, called {\it Restriction} and
{\it Frobenius}, of the form
$$
Res,F: TR^n(k;p) \to TR^{n-1}(k,p).
$$
The ordinary spectrum $TC^n(k;p)$ is the homotopy equalizer of the
maps $Res,F$.  The {\em topological cyclic homology at $p$}
of the ring $k$ is the pro-spectrum
$$
\{TC^n(k;p)\} = \left\{ \cdots \map{Res} {TC^3(k;p)}
 \map{Res} {TC^2(k;p)} \map{Res}  {TC^1(k;p)} \right\}
$$
where the transition maps are given by $Res$.

We shall need the following important theorem of Hesselholt-Madsen
\cite[7.1]{HM}. Recall that a map of $\S$-spectra $X \to Y$ is called
an {\em $\cF$-equivalence} if the induced map on fixed point spectra
$X^C \to Y^C$ is a weak equivalence of (non-equivariant) spectra
for all finite subgroups $C$ of $\S$.

\begin{thm}[Hesselholt-Madsen] \label{thm:HM}
For any ring $k$ and monoid $A$, there is a natural $\cF$-equivalence
$$
TH(k) \smsh |\Ncy(A)| \map{\sim_{\cF}}
TH(k[A])
$$ of cyclotomic spectra.
In particular, we have an equivalence of ordinary spectra
$$
TR^n(k[A],p) \sim \left(TH(k) \smsh |\Ncy(A)|\right)^{C_{p^{n-1}}},
$$
for all primes $p$ and integers $n \geq 1$.
\end{thm}

Here is an immediate consequence of the Comparison Theorem \ref{CT} and the
Hesselholt-Madsen Theorem \ref{thm:HM}. The colimit
$TR^n(k[A];p)^\fc$ is defined in \eqref{F(A)c}.

\begin{cor}
Let $k$ be a ring, let $A$ be a pctf monoid, and let $\fc=(c_1,c_2,\dots)$
be an infinite sequence of integers with $c_i \geq 2$ for all $i$.
Then we have an equivalence of ordinary spectra, natural in $k$ and $A$:
$$
TR^n(k[A];p)^\fc
\map{\sim}
\left(TH(k) \smsh |\tNcy(A)^\fc|\right)^{C_{p^{n-1}}}.
$$
\end{cor}

\section{$\tOmega$ and affine excision}\label{sec:tOmega}

The cyclic set $\tNcy (A)$ is best behaved for seminormal monoids;
see Example \ref{tNcy(cusp)}.
Since we will need to consider non-seminormal monoids in this
paper, we introduce $\tOmega_A$, a slight modification of the functor
$A\mapsto\tNcy(A)$. Recall that $A_\sn$ denotes the seminormalization
of $A$; if $A$ is a pc monoid then $A_\sn$ exists
by Proposition \ref{prop:snfunctor}.

\begin{defn}\label{def:tOmega-A}
For a  pc monoid $A$, define $\tOmega_A$ to be the $\S$-space
$$
\tOmega_A = |\tNcy(A_\sn)|.
$$
Since $A \mapsto A_\sn$ is a functor, the assignment $A\mapsto\tOmega_A$
is also functorial. Moreover, there is a natural map
$|\Ncy(A)| \to \tOmega_A$ of $\S$-spaces.
\end{defn}

%%%%%%%%%%%%%%%%%%ALL THIS IS COMMENTED OUT FOR LACK OF PROOF%%%%%%%%%%%%%%%%%%%%%%%%%%%%%%%%%
\comment{
The motivation for writing this functor as $\tOmega$ comes from the
following result (which is unproven, and not really needed in this paper):

\begin{prop} For any field $k$, let $A$ be the monoid associated to an
affine toric variety $U=\Spec k[A]$.
There is a natural isomorphism
$$
\tH_p(\tOmega_A, k) \cong \Gamma(U, \tOmega^p_{-/k})
$$
where the group on the right refers to the sheaf defined by Danilov
\cite{Dan1479}.
\end{prop}

Notice that Theorem \ref{CT}, in
conjunction with this proposition, gives a new proof of the fact that
$$
\Gamma(U_\sigma, \tOmega^p)^\fc \cong HH_p(U_\sigma)^\fc.
$$
This isomorphism was one of two key theorems needed to prove
Gubeladze's dilation conjecture in characteristic zero (\cite{chww-toric}).
 The other main
ingredient was the assertion the the cohomology of $\tOmega^p$
satisfies $cdh$ descent for schemes, which was essentially a theorem of
Danilov. We now turn our focus to proving a suitable analogue of
this $cdh$ descent result: Roughly, we prove that $\tOmega$ can be
generalized to a sheaf on pctf monoid schemes, using Zariski descent
and \ref{lem:promote}, and that it satisfies $cdh$ descent for
monoid schemes (see Theorem \ref{MT2p}).
}
%%%%%%%%%%%%%%%%%%%%%%%%%END COMMENTED MATTER%%%%%%%%%%%%%%%%%%%%%%%%%%%%

If $I$ is an ideal in a pc monoid $A$, then $\tOmega_{A/I}=|\tNcy(A_\sn/J)|$,
where $J=\sqrt{IA_\sn}$, because $A_\sn/J\to(A/I)_\sn$ is an
isomorphism by Lemma \ref{lem:newsn}. The quotient map $A\to A/I$
induces a map $\tOmega_{A} \to \tOmega_{A/I}$; we claim that this
map is onto. To see this, we may assume $A$ seminormal and $I$ radical.
Then for each $a\not\in I$ the monoid $A\ad{a}$ is seminormal
by \ref{lem:snloc1}, and $(A/I)\ad{a}=(A\ad{a})/(I\ad{a})$.
Thus the claim follows from the observation in Remark \ref{Ncy-modI}
that $\Ncy(A,a)\smap{\cong}\Ncy(A/I,a)$.

\begin{defn}\label{def:tOmega(A,I)}
Given an ideal $I$ in a pc monoid $A$,
we define $\tOmega_{A,I}$ to be the fiber over the basepoint of
$\tOmega_{A/I}$ of the surjection
$\tOmega_{A} \onto \tOmega_{A/I}$
induced by the canonical map $A \onto A/I$;
by construction, $\tOmega_{A,I}$ is the realization of a cyclic set,
and hence an $\S$-space. Since $(A/I)_\sn=(A/\sqrt{I})_\sn$ we have
$\tOmega_{A,I}=\tOmega_{A,\sqrt{I}}$ .
\end{defn}

\begin{lem}
If $A$ is a pc monoid and $I$ is an ideal, then
$$
\tOmega_{A,I} \into \tOmega_{A} \onto \tOmega_{A/I}
$$
is a split cofibration sequence of $\S$-spaces.
In particular, this sequence determines a fibration sequence of
$\S$-spectra upon smashing it with any $\S$-spectrum.
\end{lem}

\begin{proof}
Set $J=\sqrt{IA_\sn}$. By the definition of $\tNcy(A)$ and
the above remarks, there is an $\S$-equivariant decomposition
\addtocounter{equation}{-1}
\begin{subequations}
\renewcommand{\theequation}{\theparentequation.\arabic{equation}}
\begin{equation}\label{eq:tOmega(A,I)}
\tOmega_{A,I} = \bigvee_{a \in J}|\Ncy(A_\sn\ad{a}, a)|,
\end{equation}
\end{subequations}
This immediately implies that 
$\tOmega_A\cong\tOmega_{A/I} \bigvee \tOmega_{A,I}$ equivariantly.
\end{proof}

Our next result shows that $\tOmega$ satisfies ``affine excision.''

\begin{prop} \label{cofibseq}

Suppose that $f: A \to B$ is a homomorphism of seminormal pc monoids,
$I\subset A$ and $I' \subset B$ are radical ideals such that
$f$ maps $I$ bijectively onto $I'$, and the induced map
$A\ad{a} \to B\ad{f(a)}$ is an isomorphism for all $a \in I$. Then
$$
\xymatrix{
\tOmega_A \ar[r] \ar[d] & \tOmega_{A/I} \ar[d] \\
\tOmega_{B} \ar[r] & \tOmega_{B/I'} \\
}
$$
induces a homotopy cartesian square of $\S$-spectra upon smashing with
any $\S$-spectrum.
\end{prop}

\begin{proof}
It suffices to prove that the induced map
$
\tOmega_{A,I} \to \tOmega_{B,I'}
$
is an isomorphism. This is immediate
from the description \eqref{eq:tOmega(A,I)} of these spaces
and the assumptions.
\end{proof}

\begin{ex}\label{ex:closedMV}
Suppose that $I, J$ are radical ideals of a seminormal pc
monoid $A$ such that $I \cap J =0$. Taking $B=A/J$ in
Proposition \ref{cofibseq}, we see that
$$
\xymatrix{
\tOmega_A \ar[r] \ar[d] & \tOmega_{A/I} \ar[d] \\
\tOmega_{A/J} \ar[r] & \tOmega_{A/(I \cup J)} \\
}
$$
induces a homotopy cartesian square of $\S$-spectra upon smashing with
any $\S$-spectrum.
\end{ex}
\goodbreak

\begin{cor} \label{conductsOmega}
Suppose $A$ is cancellative and seminormal. Let

$I \subset A$ be the conductor ideal of the inclusion $A \into A_\nor$.
Then
$$
 \xymatrix{
 \tOmega_A \ar[r] \ar[d] & \tOmega_{A/I} \ar[d] \\
 \tOmega_{A_\nor} \ar[r] & \tOmega_{A_\nor/I} \\
 }
$$
induces a homotopy cartesian square of $\S$-spectra upon smashing it
with any $\S$-spectrum.
 \end{cor}

\begin{proof}
By Lemma \ref{lem:cond}, $I$ is a radical ideal of $A$ and $A_\nor$.
For $a \in I$, the canonical map $A \ad{a} \to A_\nor \ad{a}$
is an isomorphism, since any element of $A_\nor \ad{a}$ can be written as
$x/a^n = (ax)/a^{n+1}$ and $ax\in A$.
Now the result follows from Proposition \ref{cofibseq}.
\end{proof}

\section{Zariski and $cdh$ descent}\label{sec:descent}

In this section, we introduce the technique of descent.
Functors on monoids can be promoted to presheaves on monoid schemes
using Zariski descent, and $cdh$ descent will be used to prove our
main result, Theorem \ref{MainTheorem} below.

For any  Grothendieck
topology $t$ on the category $\cMpctf$, we may impose a model
structure on the category of presheaves of spectra on $\cMpctf$ as follows:
A morphism $\cF \to \cF'$ of presheaves is a
cofibration provided $\cF(X) \to \cF'(X)$ is a cofibration of spectra
for all $X$. A morphism is a weak equivalence provided that,
for all $n \in \Z$, the $t$-sheafification of the induced map
of presheaves of abelian groups $\pi_n
\cF \to \pi_n \cF'$ is an isomorphism:
$$
\pi_n \cF(-)^\sim_t \map{\cong}
\pi_n \cF'(-)^\sim_t.
$$
As usual, fibrations are determined by the lifting property with
respect to trivial cofibrations. The model structure comes equipped
with a functorial fibrant replacement functor, which we fix.
Given a presheaf of spectra $\cF$ on $\cMpctf$, we write $\bH_{t}(-,\cF)$
for the fibrant replacement of $\cF$ for this model structure.
Thus, $\bH_t(-, \cF)$ is a fibrant object for this model
structure and there is a natural cofibration and weak equivalence
$\cF\rightarrowtail\bH_t(-,\cF)$.
We say $\cF$ satisfies {\em $t$-descent} if $\cF(X) \to \bH_t(X, \cF)$
is weak equivalence of spectra for all $X \in \cMpctf$.

The category of presheaves of spaces on a site is similarly equipped with a closed model structure; just
replace ``spectra" by ``space" everywhere.

All of the Grothendieck topologies we are interested in will be
determined by a {\em cd structure} \cite{VVcdh}, which is a
specified family of commutative squares
\begin{equation} \label{square}
\xymatrix{
D \ar[r]_f \ar[d]^q & Y \ar[d]_p  \\
C \ar[r]^e & X
}
\end{equation}
in $\cMpctf$. The associated topology is the minimal Grothendieck
topology such that $\{e,p\}$ is a covering sieve for every square in
the cd structure.
\goodbreak

For example, the Zariski topology on $\cMpctf$ may be defined as the
topology associated to the {\em Zariski cd structure}, which by
definition consists of all squares \eqref{square} for which $Y,C$ form
an open cover of $X$, $p$ and $e$ are the canonical inclusions, and $D
= C \times_X Y = C \cap Y$. Affine monoid schemes admit no
non-trivial covers in the Zariski topology, since $\MSpec A$ has a
unique closed point.

\begin{lem}\label{lem:promote}
For any affine $U \in \cMpctf$, and any presheaf of spectra $\cF$ on $\cMpctf$,
the morphism $\cF(U) \to \bH_\zar(U,\cF)$ is a weak homotopy equivalence.
\end{lem}

\begin{proof}
Let $x$ be the unique closed point of $U$. Then $U$ represents the stalk of
the monoid scheme $U$ at $x$. Since $\cF\to\bH_\zar(-,\cF)$ is a Zariski
weak equivalence, the stalks of the map $\pi_q\cF\to\pi_q\bH_\zar(-,\cF)$
are isomorphisms. But the stalk at $x$ is the map
$\pi_q \cF(U)  \map{\cong} \pi_q \bH_\zar(U, \cF)$.
Hence the map $F(U)\to\bH_\zar(U,\cF)$ is a weak homotopy equivalence.
\end{proof}
\begin{subrem}\label{rem:promote}
Lemma \ref{lem:promote} also holds for presheaves of spaces,
with the same proof.
\end{subrem}

We also have the $cdh$ cd structure on $\cMpctf$, which consists of
the squares in the Zariski cd structure together with all
{\em abstract blow-ups}.
The latter refers to a square \eqref{square}  such that
$p:Y\to X$ is proper (in the sense of \cite[8.4]{chww-monoid}),
$e$ is the inclusion of an equivariant, closed
subscheme $C$ into $X$, $D = C \times_X Y$, and the induced map $p: Y
\setminus D \to X \setminus C$ is an isomorphism.
The associated topology is called the $cdh$ topology on $\cMpctf$.

Let $\cF$ be a presheaf of spectra defined on $\cMpctf$. We say that
$\cF$ satisfies the {\em Mayer-Vietoris property} for a square \eqref{square}
if applying $\cF$ to the square results in a homotopy cartesian square
of spectra. If a cd structure satisfies certain technical conditions
(described in \cite{VVcdh}), then a presheaf $\cF$ has the Mayer-Vietoris
property for every square in the cd structure if and only if $\cF$ satisfies descent for the associated Grothendieck topology.

\begin{ex}\label{MV=descent}
(a) The Zariski cd structure satisfies the technical conditions
of \cite{VVcdh}. Thus
$\cF$ has Zariski descent if and only if it has the Mayer-Vietoris
property for every square in the Zariski cd structure;
see \cite[Sec.\,12]{chww-monoid}. Given any presheaf of spectra $\cF$ on $\cMpctf$,
$\bH_\zar(-, \cF)$ is a presheaf of spectra on $\cMpctf$ which
satisfies the Mayer-Vietoris property
for every square in the Zariski cd structure.

(b) The $cdh$ structure also satisfies the technical conditions
of \cite{VVcdh}; see \cite[12.9]{chww-monoid}.  Thus a presheaf of spectra
$\cF$ satisfies $cdh$ descent if and only if it has the
Mayer-Vietoris property for every square in the $cdh$ structure.
Given any presheaf of spectra $\cF$ on $\cMpctf$,
$\bH_\cdh(-, \cF)$ is a presheaf of spectra which
satisfies the Mayer-Vietoris property
for every square in the cdh structure.
\end{ex}

\begin{lem}\label{colim-descent}
Let $\{ F_i\}_{i\in I}$ be a filtering system of presheaves of spectra
on $\cMpctf$. Then the natural map of presheaves
$$
\colim\bH_\zar(-,F_i) \to \bH_\zar(-,\colim F_i)
$$
is an (object-wise) weak equivalence, and similarly for the $cdh$ topology.
If each of the presheaves $F_i$ satisfies Zariski descent (resp.,
$cdh$ descent) then so does the presheaf $F=\colim F_i$.
\end{lem}

\begin{proof}
The last sentence follows from the preceding one in light of the fact
that filtering colimits of spectra commute with weak
equivalences. That the fibrant replacement commutes (up to weak
equivalence) with filtering colimits is a consequence of the fact that
the Zariski and $cdh$ sites are Noetherian, see \cite[Exp. VI, 2.11
and 5.2]{SGA4II}. In {\em loc.\ cit.}, this is proved for ordinary
cohomology of sheaves; however, since the Zariski and $cdh$ sites have
locally bounded cohomological dimension the corresponding descent
spectral sequences converge to give the desired result for the
hypercohomology spectra $\bH_\zar(-,F)$ and $\bH_\cdh(-,F)$.
\end{proof}

We will need an alternative characterization of $cdh$ descent, which
is provided by the following proposition. Recall from
\cite[Sec.\,10]{chww-monoid}
that the underlying space of a monoid scheme of finite type is a
finite poset, and that the {\em height} of a point $x$ in a monoid
scheme $X$ is the largest integer $n$ such that there exists a
strictly decreasing chain $x\!=\! x_n\!>\!\cdots \!>\! x_0$ in the poset
underlying $X$. Recall also (\cite[6.4]{chww-monoid}) that $Y\in\cMpctf$ is
{\em smooth} if each stalk is the product of a free group of finite
rank and a free monoid of finite rank.

\begin{prop} \label{prop:cdhdescent}
Suppose $\cG$ is a presheaf of spectra on $\cMpctf$ with
$\cG(\emptyset)\!\sim\! *$.
Then $\cG$ satisfies $cdh$ descent if and only if it has the
Mayer-Vietoris property for every commutative square \eqref{square}
in $\cMpctf$ that satisfies one of the following conditions:
\begin{enumerate}

\item $C=D= \emptyset$ and $Y\to X$ is the semi-normalization of $X$.

\item $C$ and $Y$ form an open covering of $X$ and $D = C \cap Y$.

\item $X$ is cancellative and seminormal, $Y \to X$ is the
  normalization of $X$, $C$ is the subscheme cut out locally by the
  conductor ideals, and $D = C \times_Y X$.

\item $X$ is seminormal, $C$ and $Y$ are equivariant closed subschemes
  of $X$ that form a closed covering of $X$, and $D = C \times_X Y$.

\item $X$ and $Y$ are connected and normal, $p$ is proper and birational,
and $C$ and $D$ are the reduced, equivariant closed subschemes of $X$ and $Y$
determined by the closures of the height one points of each.
\end{enumerate}
\noindent
Finally, if $\cG$ has $cdh$ descent and $\cG(Y)$ is contractible
for all smooth $Y$ in $\cMpctf$, then $\cG(X)$ is contractible
for all $X$ in $\cMpctf$.
\end{prop}

\begin{proof}
The proof will proceed in a number of steps which we outline now.
Step 1 shows that $cdh$ descent for a sheaf $\cG$ implies the
Mayer-Vietoris property for the types of squares listed in the statement;
the hard case is a square of type (1).
Step 2 demonstrates that if $\cG$ satisfies the Mayer-Vietoris
property for squares of types (1) through (5),
then it satisfies the Mayer-Vietoris property for a larger class
of squares that includes in particular equivariant blow-ups of monoid
schemes associated to fans; this step is the technical core of the proof.
Step 3 proves that $\cG$ satisfies $cdh$ descent
by induction on the dimension of $X$.
In the process, the final assertion
of the proposition will also be proved.

\noindent{\bf Step 1:} We first assume that $\cG$ satisfies $cdh$ descent and show
it has the Mayer-Vietoris property for each of the five types of squares.
Squares of type (2), (3) and (4) belong to the cd structure defining the $cdh$
topology and hence $\cG$ has the Mayer-Vietoris property
for each of them by \ref{MV=descent}.
Additionally, the abstract blow-up square \eqref{square} with $C=X_\red$
and $Y=D=\emptyset$ belongs to the $cdh$ structure, so
$\cG(X)\to\cG(X_\red)$ is a weak equivalence for all $X$.
For $X,Y,C,D$ as in a square of type (5),
the cartesian square given by $X,Y,C$ and $C \times_X Y$ is an
abstract blow-up and hence, again by \ref{MV=descent}, %\cite[12.9]{chww-monoid},
$\cG$ has the Mayer-Vietoris property for it. Since $D = (C
\times_X Y)_\red$, it follows that $\cG$ has the Mayer-Vietoris for
squares of type (5).

It remains to establish the Mayer-Vietoris property for squares satisfying
(1). This is equivalent to showing that
$\cG(X) \to \cG(X_\sn)$ is an equivalence for all $X$.
We proceed by induction on the Krull dimension of $X$.
Since we have seen that $\cG(X) \to \cG(X_\red)$ is an equivalence,
we may assume that $X$ is reduced.

If $\dim(X) = 0$, then $X=X_\red$ is the disjoint union of spectra
of the pointed monoids associated to abelian groups.
Hence $X_\sn=X$, and we are done.

For $\dim(X)>0$, let $C$ denote the reduced, closed equivariant
subscheme of $X$ determined by the closure of the (finitely many)
height one points of $X$. By Lemma \ref{lem:snfinite} the map
$X_\sn \to X$ is finite, and induces an isomorphism over $X\setminus C$.
Since $X, X_\sn, C$ and $D=C\times_X X_\sn$ form an abstract blow-up square,
$\cG$ has the Mayer-Vietoris property for this square,
and it suffices to show that $\cG(C)\to\cG(D)$ is a weak equivalence.
By Lemma \ref{lem:snscheme}, the map $D_\red \to C$ is the
seminormalization of $C$. Since $\dim(C) < \dim(X)$, the map
$\cG(C)\to\cG(D_\red)$ is a weak equivalence by the induction hypothesis.
Since $\cG(D)\to\cG(D_\red)$ is a weak equivalence, so is $\cG(C)\to\cG(D)$.

\noindent{\bf Step 2:} Now we assume that $\cG$ has the Mayer-Vietoris property for each of
the five types of squares appearing in the statement and prove that $\cG$ has the
Mayer-Vietoris property for any square \eqref{square} of type ($5'$),
which we define to be one that satisfies the following conditions:

\smallskip
\noindent($5'$) {\it $X$ and $Y$ are connected and normal,
$C$ is any reduced equivariant closed subscheme,
$D=C\times_X Y$, $p$ is proper and the induced map
$Y\setminus D \to X \setminus C$ is an isomorphism.}

\smallskip
To prove the Mayer-Vietoris property for such squares,
we proceed by induction on the number of points of $X\setminus C$.
If this number is $0$, there is nothing to prove since
$C=X$ and $D=Y$. If this number is $1$, then $X\setminus C$ is the
generic point of $X$ and $X, Y, C, D_\red$ form
a square of type (5); by assumption, $\cG$ has the Mayer-Vietoris
property for this square, and we are done since $D_\red=C_\sn$ by
Lemma \ref{lem:snscheme}, and
$\cG(D)\simeq \cG(D_\red)\simeq \cG(C)$ by (1).

In general, choose a maximal point $\eta$ of $X\setminus C$ and
let $C'$ denote the reduced equivariant closed subscheme
of $X$ defined by the closure of the points of $C \cup \{\eta\}$.
Setting $D' = C'\times_X Y$, we have the diagram
$$
\xymatrix{
D \ar[r] \ar[d] & D' \ar[r] \ar[d] & Y \ar[d] \\
C \ar[r] & C' \ar[r]& X.
}
$$
By the induction hypothesis, $\cG$ has the Mayer-Vietoris property
for the right-hand square. Thus it suffices to prove $\cG$ has
the Mayer-Vietoris property for the left-hand square. The scheme
$C'$ might be only partially cancellative, but it is always seminormal
(by Lemma \ref{lem:snscheme}), $\eta$ is a generic point of $C'$ and
each of the irreducible components of $C'$ is normal (by \ref{prop:pc}).
Let $C'_1$ be the irreducible component of $C'$ with generic point
$\eta$, and let $C_2$ denote the union of the other irreducible components
of $C'$. Then we have closed coverings $C'=C'_1\cup C_2$ and
$C=C_1\cup C_2$, where $C_1=C\cap C'_1$. It follows from (4) that
the right square and outer square in the following diagram
are homotopy cartesian.
$$
\xymatrix{
\cG(C')   \ar[r] \ar[d] & \cG(C)\ar[r] \ar[d] & \cG(C_2) \ar[d] \\
\cG(C'_1) \ar[r]        & \cG(C_1) \ar[r] & \cG(C_1\cap C_2)
}
$$
It follows that $\hofi(\cG(C')\to \cG(C))\sim \hofi(\cG(C'_1)\to \cG(C_1))$.

Since $Y\setminus D\cong X\setminus C$, there is a
unique point $\tilde{\eta}$ in $D'$ mapping to $\eta$. If $D'_1$ is
the irreducible component of $D'$ with generic point $\tilde{\eta}$
and $D_2$ is the union of the other components, then we have similar
closed coverings $D'=D'_1\cup D_2$ and $D=D_1\cup D_2$ where
%$D'_1$ is the irreducible component of $\tilde{\eta}$, and
$D_1=D\cap D'_1$. By the same argument, it follows from (4) that
%$\hofi(\cG(C')\to \cG(C))\sim \hofi(\cG(C'_1)\to \cG(C_1))$ and
$\hofi(\cG(D')\to \cG(D))\sim \hofi(\cG(D'_1)\to \cG(D_1))$.
Hence it suffices to show that
\[
\xymatrix{\cG(C'_1)\ar[d]\ar[r]& \cG(C_1)\ar[d]\\
          \cG(D'_1)\ar[r]& \cG(D_1)}
\]
is homotopy cartesian; this follows from (5).
\goodbreak

\noindent{\bf Step 3:}
We now assume that $\cG$ has the Mayer-Vietoris property for
all squares of type (1)--(4) and ($5'$).
In particular, $\cG$ has the Mayer-Vietoris property for
all Zariski squares as well as all squares \eqref{square} in which
$X$ is smooth, $C$ is an equivariant smooth closed subscheme and
$Y$ is the blow-up of $X$ along $C$
(these are called {\it smooth blow-up squares}).
Thus by \cite[12.12 and 12.13]{chww-monoid} we have
$$
\cG(Y) \map{\sim} \bH_\cdh(Y, \cG)
$$
for all smooth $Y$ in $\cMpctf$. That is, the homotopy fiber $\cF$
of the map $\cG \to \bH_\cdh(-,\cG)$ satisfies
$\cF(Y) \sim *$ for all smooth $Y$.
Since $\cG$ and (by what we have already proven)  $\bH_\cdh(-, \cG)$ have
the Mayer-Vietoris property for squares of type (1)--(4) and ($5'$),
so does $\cF$.

Thus it suffices to prove that if $\cF$ is a presheaf of spectra on
$\cMpctf$ that has the Mayer-Vietoris property for squares (1)--(4) and
($5'$), and is such that $\cF(Y)$ is contractible for all smooth $Y$
in $\cMpctf$, then $\cF(X)$ is contractible for all $X$ in $\cMpctf$.
In light of what we already proved, this will also establish the final assertion of the proposition.

We proceed by induction on the dimension of $X$.
Since $\cF(X) \map{\sim} \cF(X_\sn)$ by (1), we may assume $X$ is
seminormal. In particular, the closures of the minimal points of $X$
form a closed covering of $X$ by cancellative and seminormal
equivariant closed subschemes (using Lemma \ref{lem:snq} and Proposition
\ref{prop:pc}). Using the Mayer-Vietoris property for squares of type (4),
and the inductive assumption, we may
assume $X$ itself is connected, seminormal and cancellative. Using the
Mayer-Vietoris property for normalization squares of type (3), and
the induction hypothesis, we may assume $X$ is also normal.
By Example \ref{toric+pctf},   
$X$ is the toric monoid scheme associated to a fan.
In this case, by subdividing this fan, we may form an abstract blow-up square
$$
\xymatrix{
D \ar[r] \ar[d] & Y \ar[d] \\
C \ar[r] & X}
$$
such that $Y$ is smooth and $C$ is a reduced, equivariant closed
subscheme with $\dim(C) < \dim(X)$. Using the assumption that  $\cF(Y)$
is contractible, the inductive hypothesis that $\cF(C)$ and $\cF(D)$
are contractible and the Mayer-Vietoris property for squares of type ($5'$),
we conclude that $\cF(X)$ is contractible.
\end{proof}

It will be useful to weaken the hypotheses of Proposition
\ref{prop:cdhdescent} even further. Recall that each $X$ in $\cMpctf$
is separated, so that (by \cite[3.8]{chww-monoid})
the intersection of affine open subschemes is again affine open.

\begin{lem} \label{lem:cdhdescent}
Suppose $\cG$ is a presheaf of spectra on $\cMpctf$ with
$\cG(\emptyset) \sim *$. If $\cG$ satisfies Zariski descent, and
%has the Mayer-Vietoris property for all squares of type (2)
% in Proposition \ref{prop:cdhdescent} and
$\cG$ has the Mayer-Vietoris property for all squares of type
(1), (3), (4) and (5) in Proposition \ref{prop:cdhdescent}
which satisfy the additional condition that $X$ is affine, then $\cG$
satisfies $cdh$ descent.
\end{lem}

\begin{proof}
It suffices to prove $\cG$ has the Mayer-Vietoris property for all
squares \eqref{square} of type (1) and (3)--(5) in
Proposition \ref{prop:cdhdescent}.
Let $\square_X$ denote any such square, and
for an open subscheme $U \subset X$, let
$\square_U$ denote the pullback of $\square_X$ along $U \into
X$. Write $\cG(\square_U)$ for the result of applying $\cG$
to this square and taking iterated homotopy fibers.
The hypotheses imply that $\cG(\square_U)$ is contractible for any
affine open $U$, and we need to prove that
$\cG(\square_X)$ is contractible.

Since $\cG$ has the Mayer-Vietoris property for Zariski squares,
given $X = U \cup V$, the spectra $\cG(\square_X)$,
$\cG(\square_U)$, $\cG(\square_V)$, $\cG(\square_{U \cap V})$
fit together to form a homotopy cartesian square.
(This may be seen by consideration of the evident 4-dimensional cube
of spectra and properties of homotopy fibers.)
This shows that $U \mapsto \cG(\square_U)$ has the
Mayer-Vietoris property on $X_\zar$ for open covers. Since
$\cG(\square_U) \sim *$ for all affine $U$, it follows that
$\cG(\square_X) \sim *$.
\end{proof}

\section{Descent for $\tOmega$ and dilated $TC$}\label{sec:dilateddescent}

The main goal of this section is to prove Corollary \ref{Cor624} below,
that the functors $X\mapsto TC^n(X_k;p)^\fc$ satisfy $cdh$ descent.
We begin by considering the presheaf $\tOmega$.

\medskip
\paragraph{\bf Descent for $\tOmega$}
In Definition \ref{def:tOmega-A} we introduced a covariant functor
$\tOmega$ from monoids to $\S$-spaces.
We promote it to a contravariant functor from the category of monoid
schemes to $\S$-spaces by the formula
$$
(X,\cA) \mapsto \tOmega_{\cA(X)}.
$$
By Remark \ref{rem:promote}, if $X=\MSpec(A)$ is affine then the natural map
$$\tOmega_A\to\bH_\zar(X,\tOmega_{\cA})$$ is a weak homotopy equivalence.

\begin{defn}\label{def:OmegaTr}
For an $\S$-spectrum $T$ the $\S$-equivariant smash product
$\tOmega_{\cA}\smsh T$ is a presheaf of $\S$-spectra on $\cMpctf$.
We set $\tOmega^{T,0}=\tOmega_{\cA}\smsh T$
and for integers $r\ge1$, we write $\tOmega^{T,r}$ for the
presheaf of fixed-point spectra $U\mapsto(\tOmega_{\cA(U)}\smsh T)^{C_r}$
on $\cMpctf$. We consider the fibrant replacements
$\bH_\zar(-, \tOmega^{T,r})$.
\end{defn}

\begin{subrem}\label{R728}
For any affine pctf scheme $X=\MSpec(A)$, Lemma \ref{lem:promote}
implies that    $\tOmega_A\smsh T\to\bH_\zar(X,\tOmega_{\cA}\smsh T)$
and $(\tOmega_{A}\smsh T)^{C_r}\to\bH_\zar(X,\tOmega_{\cA}^{T,r})$
are weak equivalences.
\end{subrem}

\begin{thm} \label{MT2p}
For  any $\S$-spectrum $T$ and integer $r\ge0$, the presheaf of spectra
$\bH_\zar(-, \tOmega^{T,r})$
satisfies $cdh$ descent on $\cMpctf$.
\end{thm}

\begin{proof}
It suffices to check that $\bH_\zar(-, \tOmega^{T,r})$ has the
Mayer-Vietoris property for each of the five types of squares
\eqref{square} listed in Proposition \ref{prop:cdhdescent}.  Case
(2) holds by construction, since $\bH_\zar(-, \tOmega^{T,r})$
satisfies Zariski descent. For squares of type (1), (3)--(5), we may
assume $X$ is affine by Lemma \ref{lem:cdhdescent}; say $X=\MSpec
A$. Case (1) is immediate, since $X_{sn}$ and $X$ are homeomorphic
and $\tOmega_A = \tOmega_{A_\sn}$ by definition. In particular
$\bH_\zar((-)_\red, \tOmega^{T,r})\sim\bH_\zar(-, \tOmega^{T,r})$,
so we may restrict to proving the remaining cases when all schemes
in each of the squares are reduced. Cases (3) and (4) follow from
Corollary \ref{conductsOmega} and Example \ref{ex:closedMV} using
Lemma \ref{lem:promote}, since in these cases $C$, $Y$ and $D$ are
affine when $X$ is.

For case (5), assuming that $X$ is affine does not simplify the
argument, so we do not assume it. As pointed out in
Example \ref{toric+pctf}, the hypotheses imply that $X$ is a
toric monoid scheme, i.e., that $X$ is the monoid scheme associated to a
lattice $N$ and a fan $\Delta$ in $N_\R$ and $Y$ is the monoid
scheme associated to a subdivision $\tDelta$ of $\Delta$; see
\cite[4.3, 8.16]{chww-monoid}. The underlying space of $X$ is the poset of
cones in the fan $\Delta$, and $\{0\}$ is the (open) minimal point
of $X$. Moreover, $C$, $D$ are the reduced, equivariant closed
subschemes associated to the closed subsets $\Delta \setminus \{0\}$
and $\tDelta\setminus\{0\}$. As a matter of notation, for $\sigma
\in \Delta$, write $U_\sigma \subset X$ for the corresponding affine
open subscheme; by construction \cite[4.2]{chww-monoid}, $\cA_X(U_\sigma) =
(\sigma^\vee \cap M)_*$, where $M := \Hom_{\Z}(N, \Z)$.

Following Danilov \cite{Dan1479}, our next goal,
achieved in \ref{eq:Danilov-fibration}, is to understand
the homotopy fiber of
$$\bH_\zar(X,\tOmega^{T,r})\to\bH_\zar(C,\tOmega^{T,r}).
$$
%and $$\bH_\zar(Y,\tOmega^{T,r})\to\bH_\zar(D,\tOmega^{T,r}).$$
%It suffices to consider the first case; the second case is parallel.
Let $\cI \subset \cA_X$ denote the quasi-coherent
sheaf of ideals cutting out $C$ in $X$. We have
$\cI(U_\sigma) = \{m\in M\,|\, m>0\text{ on }\sigma \setminus \{0\} \}$.
Using that $\cI$ and $\cA_X$ are sheaves, it follows that for every
open subscheme $U$ of $X$, we have
$$
\cA_X(U) = \{m \in M \, | \, \text{$m \geq 0$ on $|U|$} \}_*
$$
and
$$
\cI(U) = \{m \in M \, | \, \text{$m > 0$ on $|U| \setminus \{0\}$} \}_* ,
$$
where $|U|$ denotes the closed subset of $N_\R$ given as union of the
cones in the set $U\subset X = \Delta$.
For example, if $U=\{0\}$ then $|U|=\{0\}$ and $\cI(U)=\cA(U)=M$.

Let $X_\zar$ be the category whose objects are the open monoid subschemes
of $X$ and whose morphisms are the inclusions.
We have the presheaf of spectra on $X_\zar$,
$$
U \mapsto \left( \tOmega_{\cA_X(U), I(U)} \smsh T \right)^{C_r},
$$
where $\tOmega_{A,I}$ is defined in \ref{def:tOmega(A,I)},
and we define the spectrum
$$
\bH_\zar((X,C), \tOmega^{T,r}) := \bH_\zar\left(X, U \mapsto
\left( \tOmega_{\cA_X(U), I(U)} \smsh T \right)^{C_r}\right).
$$
There is a sequence of maps of presheaves of spectra on $X_\zar$
given on an open $U \subset X$ as
\addtocounter{equation}{-1}
\begin{subequations}
\renewcommand{\theequation}{\theparentequation.\arabic{equation}}
\begin{equation}\label{eq:presheaf-fibration}
\left( \tOmega_{\cA_X(U),    I(U)} \smsh T \right)^{C_r}
\to \left( \tOmega_{\cA_X(U)} \smsh T \right)^{C_r}
\to \left( \tOmega_{\cA_X(U)/  I(U)} \smsh T \right)^{C_r}.
\end{equation}
Since each $\cA(U)$ is seminormal,
$\tOmega_{\cA_X,I}\to\tOmega_{\cA_X}\to\tOmega_{\cA_X/I}$ is a
presheaf of cofibration sequences by Proposition \ref{cofibseq},
which implies that \eqref{eq:presheaf-fibration} is a fibration sequence
for each $U$. Applying $\bH_\zar(X, -)$ therefore yields
a fibration sequence of spectra
\begin{equation}\label{eq:Danilov-fibration}
\bH_\zar((X,C), \tOmega^{T,r})
\to \bH_\zar(X, \tOmega^{T,r}) \to \bH_\zar(C, \tOmega^{T,r}).
\end{equation}
This fibration is natural in $(X,C)$.
Replacing $(X,C)$ by $(Y,D)$ gives an analogous fibration sequence %for $Y,D$
and a map
\begin{equation} \label{E630}
\bH_\zar((X,C), \tOmega^{T,r}) \to \bH_\zar((Y,D), \tOmega^{T,r}).
\end{equation}
To prove Theorem \ref{MT2p} it remains to prove the map \eqref{E630}
is a weak equivalence.

Our next goal, achieved in \eqref{eq:m-factor}, is to
decompose the map \eqref{E630} into a wedge sum of maps indexed by $M$.
By \eqref{eq:tOmega(A,I)}, we have a decomposition
into presheaves of $\S$-spaces on $X_\zar$:
\begin{equation} \label{E630b}
\tOmega_{\cA_X(U), \cI(U)} = \bigvee_{m \in \cI(U)}
|\Ncy(\cA_X(U)[-m], m)|.
\end{equation}
\noindent (Because we are using additive notation, we write $A[-m]$
instead of $A\ad{m}$.)

We claim that if $m \in \cI(U)\setminus\{0\}$, then $\cA_X(U)[-m]=M_*$.
To see this, note that for each $\sigma \in U$, we have $m>0$ on
$|\sigma| \setminus \{0\}$ and hence (provided $\sigma \ne \{0\}$),
$m$ lies in the interior of $\sigma^\vee$
and thus $(\sigma^{\vee} \cap M)[-m] = M$.
That is, for each $l \in M$,
there is a positive integer $N_\sigma$ such that $l + N_{\sigma} m \in
\sigma^\vee$. Since $U$ consists of a finite number of cones, we may
find a positive integer $N$ such that $l + N m \in \sigma^\vee$ for all
$\sigma \in U$ and hence $l + Nm \in \cA_X(U)$.
This shows $l\in\cA_X(U)[-m]$.

Using this and $|\Ncy(\cA_X(U),0)|=*$, \eqref{E630b} becomes
the $\S$-equivariant decomposition
$$
\tOmega_{\cA_X(U), \cI(U)} = \bigvee_{m \in M} G_m^N(U),
$$
where $G_m^N$ is the presheaf of $\S$-spaces on $X_\zar$ defined by
$$
G_m^N(U) =
\begin{cases}
|\Ncy(M_*,m)|, & \text{if $m > 0$ on } |U| \setminus \{0\}, \\
\ast, & \text{otherwise.}
\end{cases}
$$
Since $(\bigvee G_m^N)\smsh T \map{\sim} \bigvee(G_m^N\smsh T)$ and
$(\bigvee G_m^N\smsh T)^{C_r}=\bigvee(G_m^N\smsh T)^{C_r}$, it follows that
$\bH_\zar((X,C),\tOmega^{T,r})=\bH_\zar(X,\bigvee(G_m^N\smsh T)^{C_r})$.
Set $G_m=(G_m^N\smsh T)^{C_r}$.

By Lemma \ref{colim-descent},  $\bigvee \bH_\zar(-,G_m)$ satisfies
Zariski descent and
\[
\bH_\zar((X,C), \tOmega^{T,r}) = \bH_\zar(X,\bigvee G_m)
\map{\sim} \bigvee_{m \in M}\!\!\bH_\zar(X,G_m)
\]
is a weak equivalence.
Similarly, we have an equivalence
\[
\bH_\zar((Y,D),\tOmega^{T,r}) = \bH_\zar(X,\bigvee G_m) \map{\sim}
\bigvee_{m\in M}\!\!\bH_\zar(Y,G_m)
\]
The evident square commutes, so the map \eqref{E630} decomposes as a
wedge sum of maps
\begin{equation}\label{eq:m-factor}
\bH_\zar(X,G_m) \longrightarrow \bH_\zar(Y,G_m).
\end{equation}
Thus it suffices to show that for all $m\in M$, the map \eqref{eq:m-factor}
is an equivalence. This is done in Lemma \ref{L312} below, with
$E=(|\Ncy(M,m)|\smsh T)^{C_r}.$
\end{subequations}
\end{proof}

Before stating the lemma that was used in the above proof, we
introduce a simple construction which we will use.

\begin{construction}\label{const:C(m)}
Fix a lattice $N$, an isomorphism $N\cong\Z^n$ and an element
$m$ of $\Hom(N,\R)$. We set
\[
C(m) := \{x \in N_\R | \, m(x) \leq 0, \|x\| \ge1\} \subset N_\R.
\]
For each fan $\Delta$ in $N_{\R}$, with underlying cone
$|\Delta|\subseteq N_\R$, we let $K(\Delta)$ denote the quotient
$K(\Delta)=|\Delta|/(C(m)\cap|\Delta|)$. By convention,
if $|\Delta|$ is disjoint from $C(m)$, then $K(\Delta)$ is
$|\Delta|$ with a disjoint basepoint adjoined.  If $\Delta$ is affine,
then $|\Delta|$ is a strongly convex rational polyhedral cone,
so $K(\Delta)$ is contractible (because the intersections of both
$|\Delta|$ and $C(m)\cap|\Delta|$ with any sphere are convex on the
sphere). If $\Delta$ is the union of two fans $\Delta_1$ and $\Delta_2$
then $K(\Delta_1)$ and $K(\Delta_2)$ form a closed cover of $K(\Delta)$
whose intersection is $K(\Delta_1\cap\Delta_2)$; one can even choose
a CW structure on $|\Delta|$ so that the $|\Delta_i|$ and the
$C(m)\cap|\Delta_i|$ are subcomplexes and the cover is cellular.
\end{construction}

We now state and prove the lemma that was used in the proof of
Theorem \ref{MT2p}.
%In what follows, we let $\Maps_*(S,E)$ denote the function spectrum
%associated to a pointed space $S$ and a spectrum $E$.
% and recall the convention that for a space $S$, $S/\emptyset$ is
% $S$ with a disjoint basepoint adjoined. Also we will, as usual,
% write $M$ for $\Hom_\Z(N,\Z).$

\begin{lem} \label{L312}
Let $X=(\Delta,N)$ be a toric monoid scheme and let $Y=(\Delta',N)$
be another toric monoid scheme with the same lattice and such that
$\Delta'$ is a refinement of $\Delta$. Let $f: Y\to X$ be the
associated morphism of monoid schemes (see \cite[4.2]{chww-monoid}).
Fix a (non-equivariant) spectrum $E$ and an element $m \in M=\Hom(N,\R)$.
Let $G_m^E$ denote the presheaf of spectra on $X_\zar$ (and $Y_\zar$,
respectively) defined by
$$
U \mapsto G_m^E(U) =
\begin{cases}
E & \text{if $m>0$ on $|U| \setminus\{0\}$}
\\ * & \text{otherwise.} \\
\end{cases}
$$
Then the map $f^*: \bH_\zar(X,G_m^E) \to \bH_\zar(Y,G_m^E)$
induced by $f$ is a weak equivalence.
\end{lem}

\begin{proof}
We give %The proof works by giving
an explicit description (up to weak equivalence) of the presheaves
$\bH_\zar(-,G_m^E)$ whose global sections only depend on $|\Delta|$.

For $U$ an open in $X$ (or $Y$), we regard $U$ as a fan in $N_\R$
and write $F(U)$ for $\Maps_*(K(U),E)$,
where $K(U)$ is defined in Construction \ref{const:C(m)} and
$\Maps_*(K,E)$ denotes the function spectrum associated to a
pointed space $K$ and a spectrum $E$.
Because $K(U)$ is a contravariant functor of $U$,
$F$ is a presheaf on $X$, and on $Y$.
Then we assert:

(a)
$F(X)\to\bH_\zar(X,F)$ is a weak equivalence, and similarly for $Y$;

(b) $\bH_\zar(X,G_m^E) \map{\sim} \bH_\zar(X,F)$ is an equivalence of spectra,
 and similarly for $Y$;

(c) the following diagram commutes:
$$
\xymatrix{
\bH_\zar(X, G_m^E) \ar[r]^{\phantom{X}\sim} \ar[d]^{f^*} &\bH_\zar(X,F)\ar[d]
\\
\bH_\zar(Y, G_m^E) \ar[r]^{\phantom{X}\sim} &\bH_\zar(Y,F).}
$$
Since $F(X) = F(Y)$ by the very definition of $F$,
this will prove the lemma.

To establish assertion (a), it suffices (by \ref{MV=descent})
to show that the presheaf $F$ has the Mayer-Vietoris property for open covers.
If $U = V \cup W$ is an open cover, we saw in  \ref{const:C(m)} that
$K(V)$ and $K(W)$ form a closed cellular cover of $K(U)$.
%
%one can choose a CW complex structure on $|U|$ such that
%the closed subsets $|V|$ and $|W|$ are subcomplexes and the
%intersections with $C(m)$ are subcomplexes as well (e.g.
%because each of these spaces is a closed semi-algebraic subset of $N_\R$).
%Hence $|V|/(C(m) \cap |V|)$ and
%$|W|/(C(m) \cap |W|)$ form a closed cover of $|U|/(C(m)\cap|U|)$
%whose intersection is $|V\cap W|/(C(m)\cap|V\cap W|)$.
%
Since $\Maps_*(-,E)$ has descent for closed cellular covers (by the
homotopy extension property for embeddings of subcomplexes), we see that
$F$ has the Mayer-Vietoris property, establishing (a).

For (b), we define a morphism $G_m^E\to F$ of presheaves of spectra
on $X$ (or $Y$, respectively) by the following formula:
when $C(m) \cap |U| \ne \emptyset$ so that $G_m^E(U) = *$, it
is the inclusion of the basepoint; when $C(m)\cap|U|=\emptyset$, it is
the map
$$
G_m^E(U) = E = \Maps_*(S^0, E) \to \Maps_*(|U|/\emptyset, E) = F(U)
$$
induced by the continuous function $|U|/\emptyset \onto
{1}/\emptyset = S^0$ that sends $|U|$ to ${1}$.

We claim that $G_m^E(U)\to F(U)$ is a homotopy equivalence
for all affine $U$ in $X_\zar$. Indeed,
if $C(m)\cap|U|=\emptyset$ then the pointed space $K(U)$
%$|U|/(C(m)\cap|U|)$
is homotopy equivalent to $S^0$, and the claim is clear.
On the other hand, If $C(m)\cap|U|\ne\emptyset$ then $G_m^E(U)=*$;
when $U$ is affine, the space $K(U)$
%$|U|/(C(m)\cap|U|)$
is contractible, as noted in \ref{const:C(m)},
and $F(U)$ is contractible. This establishes the claim.

The claim implies that $G_m^E\to F$ is (locally in the Zariski
topology) a weak equivalence of presheaves of spectra, and hence
\[
\bH_\zar(X, G_m^E) \map{\sim}\bH_\zar(X,F)
\]
is also a weak equivalence of spectra (and similarly for $Y$),
which is assertion (b).

Note that the presheaves $F$ and $f_* F$ on $X_\zar$ are, in fact, equal.
Assertion (c) is an immediate consequence of the
functoriality of fibrant replacements.
\end{proof}

\medskip
\paragraph{\bf Descent for dilated $TC$}
Given a commutative ring $k$, the $k$-realization $X_k$ of a
monoid scheme $X$ is a scheme over\ $\Spec(k)$, constructed in
\cite[5.3, 5.9]{chww-monoid}. It is functorial in $X$, and if $X=\MSpec A$ is
affine, then $X_k = \Spec(k[A])$. If $X$ is the toric monoid scheme
associated to a fan $(N,\Delta)$ as in Example \ref{toric+pctf},
and $k$ is a field,
then $X_k$ is the usual toric $k$-variety associated to this fan.

Recall from \eqref{def:TR(k)} that $TR^n(-;p)$ is a covariant functor
from commutative rings to $\S$-spectra. We promote this to a presheaf
$TR^n(\cO;p)$ on schemes by sending $(Y,\cO_Y)$ to
$TR^n(\cO_Y(Y);p)$. Following Geisser and Hesselholt \cite[3.3]{GH97},
we define the presheaf $TR^n(-;p)$ on schemes by
$$
TR^n(Y;p)  = \bH_{\zar/k}(Y,TR^n(\cO;p)),
$$
where $\bH_{\zar/k}$ is the fibrant replacement for
presheaves defined on schemes of finite type over $k$.
By \cite[3.3.3]{GH97}, $TR^n(S;p) \sim TR^n(\Spec S;p)$
for any ring $S$.

Composing with the $k$-realization functor 
turns $TR^n(-;p)$ into a presheaf on the category of monoid schemes.
Given an open covering of a separated monoid scheme $X$ by $U$ and $V$,
applying the $k$-realization functor yields an open covering of $X_k$
by $U_k$ and $V_k$, with $U_k\cap V_k=(U\cap V)_k$; see \cite[5.3]{chww-monoid}.
Using \ref{MV=descent}(a), it follows that
$TR^n(-;p)$ satisfies Zariski descent on $\cMpctf$.
Given an infinite sequence $\fc$ of integers $\ge2$, the presheaf
$TR^n(-;p)^\fc$ also satisfies Zariski descent by
Lemma \ref{colim-descent}.

\begin{thm} \label{Thm625}
Let $k$ be a ring and $\fc$ an infinite sequence of
integers $c_1, c_2, \dots$ with $c_i \geq 2$ for all $i$.
Then for all primes $p$ and integers $n \geq 1$, the presheaf
$$
X \mapsto TR^n(X_k;p)^\fc
$$
satisfies $cdh$ descent on $\cMpctf$.
\end{thm}

\begin{proof}%[Sketch of Proof]
Since $TR^n(-;p)^\fc$ satisfies Zariski descent on $\cMpctf$,
the Hesselholt-Madsen Theorem \ref{thm:HM} implies that
there is an equivalence
$$
\begin{aligned}
TR^n(X_k;p)^\fc & \sim \bH_\zar(X,TR^n(-;p)^\fc) \\
%\bH\left(X, U \mapsto
%\left(TH(k) \smsh |\Ncy(\cA_X(U)) |\right)^{C_{p^{n-1}}}    \right)^\fc \\
& \sim_{\cF}
\bH\left(X, U \mapsto
\left(TH(k) \smsh |\Ncy(\cA_X(U))|^\fc\right)^{C_{p^{n-1}}}    \right). \\
\end{aligned}
$$
Here we have used the fact that the filtered colimit $(-)^\fc$
commutes with finite limits such as $(-)^{C_r}$ and with smashing
with a spectrum $T$. By Theorem \ref{CT} and Definition
\ref{def:tOmega-A}, we have a natural equivalence of $\S$-spectra
$|\Ncy(A)|^\fc \simeq |\tNcy(A)_\sn|^\fc = \tOmega_A^\fc$. Replacing
$A$ by $\cA_X(U)$, smashing with $TH(k)$, taking $C_r$-fixed points
with $r=p^n$ and then applying $\bH_\zar$, we obtain the equivalence
$$
\begin{aligned}
TR^n(X_k;p)^\fc & \sim
\bH_\zar\left(X,(TH(k)\smsh\tOmega_{\cA}^\fc)^{C_r}\right) \cong
\bH_\zar\left(X, (\tOmega^{TH(k),r})^\fc\right).
\end{aligned}
$$
(The final $\cong$ uses Definition \ref{def:OmegaTr}.)
By Theorem \ref{MT2p}, $\bH_\zar(-,\tOmega^{T,r})$ satisfies $cdh$ descent.
The result now follows from Lemma \ref{colim-descent}.
\end{proof}

\begin{cor} \label{Cor624}
For any ring $k$ and integer $n\ge1$, the spectrum-valued functor
$$
X \mapsto TC^n(X_k;p)^\fc
$$
satisfies $cdh$ descent on $\cMpctf$.
\end{cor}

\begin{proof}
Recall that $TC^n(X_k;p)$ is the homotopy equalizer of the two maps
\[
TR^n(X_k; p) \twomaps TR^{n-1}(X_k; p)
\]
given by restriction and Frobenius. We may identify
$TC^n(X_k;p)^\fc$ as either the colimit of the sequence of
endomorphisms of the spectra $TC^n(X_k;p)$ by the map $\theta_{c_i}$
or as the homotopy equalizer of the induced maps from
$TR^n(X_k; p)^\fc$ to $TR^{n-1}(X_k; p)^\fc$.
Since a homotopy pullback of presheaves satisfying $cdh$ descent
satisfies $cdh$ descent, the assertion now follows from Theorem \ref{Thm625}.
\end{proof}

\section{The Dilation theorem in characteristic $0$}\label{sec:char.0}

\def\HH{\operatorname{HH_\cdot}}
\def\HHcy{\operatorname{HH^{cy}_\cdot}}

In our previous paper \cite{chww-toric}, we proved that the canonical map
$$
\cK(X_k)^\fc \to \cKH(X_k)^\fc
$$
is a weak equivalence of spectra whenever
$k$ is a field of characteristic zero and $X$ is a toric monoid scheme
(so that $X_k$ is a toric variety).
In this section, by using the results developed in this paper, we
extend this result slightly to include all $X$ in
$\cMpctf$. (Our result here also applies to regular rings $k$
containing $\Q$.) In the next section, we
will prove this result in the more difficult case when the
characteristic of $k$ is positive.

For a commutative ring $k$ we shall write $H_k$ for the
Eilenberg-Mac\,Lane spectrum associated to $k$.
Given a $k$-algebra $R$, we shall write $\HH(R/k)$ for the
generalized Eilenberg-Mac\,Lane spectrum associated
to the standard Hochschild complex
for the $k$-algebra $R$.
Thus $\HH(-/k)$ is a covariant functor from $k$-algebras to spectra such
that $\pi_n(\HH(R/k))$ is $HH_n(R/k)$, the $n$-th Hochschild homology group
of the $k$-algebra $R$.
Recall from \eqref{eq:H(N)=HH} that,
for any monoid $A$ and ring $k$, we have a natural weak equivalence
of spectra
$$
\Ncy(A) \smsh H_k \sim k[\Ncy(A)] = \HH(k[A]/k),
$$
and a natural isomorphism
$
H_q(\Ncy(A), k) \cong HH_q(k[A]/k).
$

As $\Ncy(A) \smsh H_k$ is a functor from monoids to spectra, we may
promote it to a functor from monoid schemes to spectra by sending $(X,\cA)$
to $\Ncy(\cA(X)) \smsh H_k$. Let us write this functor as $\HH(k[\cA])$.
Taking fibrant replacements for the Zariski topology, we obtain the
presheaf of spectra
$$
X \mapsto \bH_\zar(X,\HH(k[\cA]))
$$
defined on monoid schemes. 
If $X = \MSpec(A)$ is an affine monoid scheme, the above remarks show
there is a natural weak equivalence of spectra
$$
\HH(k[A]/k) \map{\sim}\bH_\zar(X,\HH(k[\cA])).
$$
As in \eqref{F(A)c}, given a sequence $\fc = \{c_1,c_2,\dots\}$
of integers with $c_i\ge2$ for all $i$, taking colimits yields
the presheaf of spectra %on $\cMpctf$.
$
\bH_\zar(-,\HH(k[\cA]))^\fc.
$
For any $X$ in $\cMpctf$, we have
\begin{equation} \label{E428}
\bH_\zar(X, \HH(k[\cA]))^\fc
\map{\sim}
\bH_{\zar/k}(X_k, \HH(-/k))^\fc,
\end{equation}
where $\bH_{\zar/k}$ denotes fibrant replacement for the Zariski
topology on the category of schemes of finite type over $k$.

\begin{thm}  \label{cdh-descent-0}
If $k$ is a regular ring containing  $\Q$
and $\fc = \{c_1, c_2, \dots \}$ is a sequence of integers with $c_i\ge2$
for all $i$, then the presheaves of spectra $\bH_\zar(-,\HH(k[\cA]))^\fc$
and $\bH_{\zar/k}(-_k,\HH(-/k))^\fc$
satisfy $cdh$-descent on $\cMpctf$.
\end{thm}

\begin{proof} By Theorem \ref{CT}, it suffices to prove that the
analogously defined presheaf defined using $\tNcy$ in place of
$\Ncy$ satisfies $cdh$-descent. This holds by Theorem \ref{MT2p},
letting $T$ be the Eilenberg-Mac\,Lane spectrum $H_k$, regarded as a
$\S$-spectrum with trivial action,
so that $\tOmega^{T,0}(\MSpec(A))$ is $|\tNcy(A_\sn)|\smsh H_k$.
\end{proof}

\begin{subrem}
An equivariant spectrum is indexed by finite sub-representations of an
$\S$-universe $U$, while an ordinary spectrum is indexed by finite
dimensional subspaces of $U^{\S}$. Thus in order to regard an ordinary
spectrum $H$ as an equivariant spectrum with trivial action,
one needs to extend the indexing family. This is accomplished by
the left adjoint of the forgetful functor from equivariant spectra
or ordinary spectra. We shall not dwell on this standard construction.
\end{subrem}

Now let $\H_{\cdh/k}$ denote fibrant replacement for the $cdh$ topology
on schemes of finite type over $k$. We are interested in the canonical map
\begin{equation} \label{E425}
\bH_{\zar/k}(X_k, \HH(-/\Q))^\fc
\to \bH_{\cdh/k}(X_k, \HH(-/\Q))^\fc
\end{equation}

\begin{prop}\label{Qdescent}
If $k=\Q$, \eqref{E425} is a weak equivalence for all $X$ in $\cMpctf$.
\end{prop}

\begin{proof}
Let us write $\cF(X)$ for the source of \eqref{E425},
regarding $\cF$ as a presheaf on $\cMpctf$.  Using \eqref{E428},
Theorem \ref{cdh-descent-0}
shows that $\cF$ satisfies $cdh$ descent on $\cMpctf$.
Since $\Q$-realization sends $cdh$ squares in $\cMpctf$ to
$cdh$ squares of schemes of finite type over $\Q$, the target of
\eqref{E425} also satisfies $cdh$-descent as a presheaf on $\cMpctf$.

Since monoid schemes are locally smooth and affine for the $cdh$ topology
by \cite[11.1]{chww-monoid}, we may assume that $X$ is smooth and affine.
In this case, $X_\Q$ is smooth over $\Q$ by \cite[6.4--5]{chww-monoid}.
(In fact, $X$ is a finite product of copies of $\A^1$ and $\A^1-\{0\}$.)
%product of a free group of finite rank and a free monoid of finite rank.)
Moreover, $\HH(X_\Q/\Q))^\fc\map{\sim}\cF(X)$ is a weak equivalence.

There is a notion of $scdh$ descent for presheaves on $Sm/\Q$,
and if a presheaf $G$ satisfies $scdh$ descent then $\bH_\cdh(-,G)\simeq\bH_\scdh(-,G)$
(see the argument preceding Theorem 2.4 in \cite{chww-vorst}).
%and for smooth monoid schemes, given in \cite[12.12]{chww-monoid}.
% $\Q$-realization sends $scdh$ squares to $scdh$ squares in $Sm/\Q$.
It is proven in \cite[2.9, 2.10, 3.9]{chsw} that
$\bH_{\zar/\Q}(-,\HH(-/\Q))$ satisfies $scdh$ descent on $Sm/\Q$.
Since $X_\Q$ is smooth affine, this implies that the maps
\addtocounter{equation}{-1}
\begin{subequations}
\renewcommand{\theequation}{\theparentequation.\arabic{equation}}
\begin{equation}\label{eq:cdh-scdh}
\HH(X_\Q/\Q)\map{\sim}\bH_{\cdh/\Q}(X_\Q,\HH(-/\Q))
\map{\sim}\bH_{\scdh/\Q}(X_\Q,\HH(-/\Q))
\end{equation}
\end{subequations}
are weak equivalences. Therefore $\HH(X_\Q/\Q)^\fc$
is weakly equivalent to the target of \eqref{E425}, as required.
\end{proof}

%Hence $\cF$ also satisfies
%$scdh$ descent on the category of smooth monoid schemes.
%When $X$ is a smooth monoid scheme, we have weak equivalences
%$\cF(X)\map{\sim}\bH_\cdh(X,\cF)\map{\sim}\bH_\scdh(X,\cF)$
%by \cite[12.14]{chww-monoid}. Since
%$\bH_{\cdh/\Q}(X_\Q,\HH(-/\Q))^\fc\map{\sim}\bH_{\cdh/\Q}(X_\Q,\cF)$
%is a weak equivalence, we are done.

\begin{thm}\label{dilation-0}
If $k$ is a regular ring containing  $\Q$
and $\fc = \{c_1, c_2, \dots \}$ is a sequence of integers with $c_i\ge2$
for all $i$, then the canonical map
%With $k$ and $\fc$ as in Theorem \ref{cdh-descent-0},
$$
\cK(X_k)^\fc \to \cKH(X_k)^\fc
$$
is a weak equivalence for all $X \in \cMpctf$.
In particular, if $A$ is a cancellative and
torsionfree monoid with no non-trivial units, then
$K_*(k[A])^\fc\cong K_*(k)$ and
$$
\cK(k[A])^\fc \sim \cK(k).
$$
\end{thm}

When $k$ is a field with $\chr(k) = 0$, this theorem was proved
in \cite{chww-toric}. The isomorphism $K_*(k[A])^\fc\cong K_*(k)$
when $k$ is regular is due to Gubeladze \cite{Gu08}.

\begin{proof}
The proofs of Corollary 6.8 and Theorem 6.9 of our
previous paper \cite{chww-toric} apply verbatim to show that
$\cK(X_k)^\fc \to \cKH(X_k)^\fc$ is a weak equivalence
if the canonical map \eqref{E425} is a weak equivalence of spectra.
The proof that this implies that $\cK(k[A])^\fc \sim \cK(k)$
is given in \cite[6.10]{chww-toric}; a shorter proof is given in
Corollary \ref{Gconjecture} below.

The source of \eqref{E425} may be understood using the following device.
By the K\"unneth formula, we have a natural weak equivalence
$$
\HH(R/\Q) \smsh \HH(k/\Q) \map{\sim}
\HH(R \otimes_\Q k/\Q)
$$
for any $\Q$-algebra $R$. Applying this locally,
% to $X_k = X_\Q \times_{\Spec \Q} \Spec k$,
we see that the canonical map
$$
\bH_{\zar/\Q}(Y,\HH(-/\Q))^\fc \smsh \HH(k/\Q)
\map{\sim}
\bH_{\zar/k}(Y_k, \HH(-/\Q))^\fc
$$
is a weak equivalence for every
noetherian scheme $Y$ over $\Q$, including $Y=X_\Q$.

Since $k$ is a filtered union of smooth $\Q$-algebras of finite type
(by Popescu's theorem), and $K$-theory commutes with filtered colimits,
we may also assume that $k$ is finitely generated smooth over $\Q$.
Since the basechange $-\otimes_{\Spec\Q}\Spec k$
sends $cdh$ squares to $cdh$ squares, an arrow $\mu$ exists
making the following diagram commutative up to homotopy.
\[ \xymatrix{
\bH_{\zar/\Q}(X_\Q, \HH(-/\Q))^\fc\smsh \HH(k/\Q)
 \ar[r]^{\phantom{XXXX}\simeq}   \ar[d]^{\simeq}_{\ref{Qdescent}}  &
\bH_{\zar/k}(X_k, \HH(-/\Q))^\fc \ar[d]^{\eqref{E425}} \\
\bH_{\cdh/\Q}(X_\Q, \HH(-/\Q))^\fc \smsh \HH(k/\Q)
\ar[r]^{\phantom{XXXX}\mu} &
\bH_{\cdh/k}(X_k, \HH(-/\Q))^\fc
}\]
The left arrow is a weak equivalence by Proposition \ref{Qdescent},
and the top arrow is a weak equivalence by the above remarks.
Thus it suffices to show that the bottom arrow $\mu$ is a weak equivalence.

%$$
%\bH_{\cdh/\Q}(X_\Q, \HH(-/\Q))^\fc \smsh \HH(k/\Q)
%\to
%\bH_{\cdh/k}(X_k, \HH(-/\Q))^\fc.
%$$
%We claim this is also a weak equivalence.

Since $k$-realization sends
$cdh$ squares in $\cMpctf$ to $cdh$ squares of schemes of finite type
over $k$, both the source and target of $\mu$ satisfy $cdh$-descent
on $\cMpctf$. Since every $X$ in $\cMpctf$ is locally isomorphic in
the $cdh$ topology to a smooth affine monoid scheme,
we may assume $X$ is such a scheme.
Because $k$ is smooth over $\Q$, the map
$$
\HH(X_k/\Q) \to
\bH_{\cdh/k}(X_k, \HH(-/\Q))
$$
is a weak equivalence. We saw in \eqref{eq:cdh-scdh} that
$\HH(X_\Q/\Q)\to\bH_{\cdh/\Q}(X_\Q,\HH(-/\Q))$
is a weak equivalence, so $\mu$ is weak equivalent to the map
\[
\HH(X_\Q/\Q)^\fc\smsh \HH(k/\Q) \map{} \HH(X_k/\Q)^\fc.
\]
Since $X_\Q$ is affine, this is a weak equivalence by the K\"unneth formula.
\end{proof}

\section{The Dilation theorem}\label{sec:maintheorem}

We recall some basic facts about pro-objects from \cite{Artin-Mazur}.
A {\em pro-abelian group} is a sequence of homomorphisms
  of abelian groups indexed by the positive integers,
$$
\cdots \to A^3 \to A^2 \to  A^1,
$$
or, in other words, it is a contravariant functor from $\N$ to the category of
abelian groups, where $\N$ is the ordered set of positive integers
viewed as a category in the standard way. We typically write such a
pro-abelian group as $\{A^n\}$.
The category $\ProAb$ has as objects all pro-abelian groups and
morphisms defined by
$$
\Hom_{\ProAb}(\{A^n\}, \{B^n\}) =
\varprojlim_{m\in\N} \varinjlim_{n\in\N} \Hom_\Ab(A^n, B^m).
$$
The category $\ProAb$ is an abelian category.

A {\em strict morphism} of pro-abelian groups will refer to a natural
transformation of functors from $\N$ to abelian groups. The collection
of pro-abelian groups with arrows defined by strict morphisms
is the abelian category $\Ab^{\N}$ of contravariant functors
from $\N$ to abelian groups. The evident functor from $\Ab^{\N}$
to $\ProAb$ is exact and preserves all finite limits and colimits
\cite[A.4.1]{Artin-Mazur}.
In particular, if we consider the morphism in $\ProAb$ associated to a
strict morphism, its kernel and cokernel are also given degree-wise.

\begin{lem}\label{lem:prozero}
A pro-abelian group $\{A^n\}$ is isomorphic to zero in $\ProAb$
if and only if for all $m$ there is an $n\ge m$
such that $A^n \to A^m$ is the zero map.
\end{lem}

\begin{proof}
For each $m$, let $g_m \in \varinjlim_n\Hom_\Ab(A^n, A^m)$ denote
the image of $A^m \map{\text{id}} A^m$ under the canonical map
$\Hom_\Ab(A^m,A^m) \to \varinjlim_n \Hom_\Ab(A^n,A^m)$.
Then the identity map of $\{A^n\}$ is represented by the element
$(g_m)_{m \in \N}$ of the inverse limit of the
$\{ \varinjlim_n \Hom_\Ab(A^n,A^m) \}$.
Clearly, $\{A^n\}\cong0$ if and only if the identity map and the zero map
coincide in $\Hom_{\ProAb}(\{A^n\}, \{A^n\})$. This is equivalent to
the condition that all the $g_m$ are zero.
%The former is given by the system $(g_m)_{m \in \N}$ of maps, where
%$g_m \in \varinjlim_n\Hom_\Ab(A^n, A^m)$ is the image of
%$\text{id}: A^m \to A^m$ under the canonical map
%$\Hom_\Ab(A^m,A^m) \to \varinjlim_n \Hom_\Ab(A^n,A^m)$.
On the other hand, $g_m = 0$ in $\varinjlim_n \Hom_\Ab(A^n, A^m)$
if and only if there exists an $n \geq m$ so that $A^n \to A^m$ is the
zero map.
\end{proof}

%Clearly, a pro-abelian group $\{A^n\}$ is (isomorphic to) the zero
%object if and only if the identity map and the zero map in
%$\Hom_{\ProAb}(\{A^n\}, \{A^n\})$ coincide. The former is given by the
%tuple of maps $(g_m)_{m \in \N}$ where
%$g_m \in \varinjlim_n\Hom_\Ab(A^n, A^m)$ is the image of
%$\text{id}: A^m \to A^m$ under the
%canonical map $\Hom_\Ab(A^m, A^m) \to \varinjlim_n \Hom_\Ab(A^n,
%A^m)$.  Observe that $g_m = 0$ in $\varinjlim_n \Hom_\Ab(A^n, A^m)$
%if and only if there exists an $n \geq m$ so that $A^n \to A^m$ is the
%zero map.  Thus, $\{A^n\}$ is the zero object in $\ProAb$ if and only
%if for all $m$ there is an $n \geq m$ such that $A^n \to A^m$ is the
%zero map.

\begin{lem}\label{lem:constantpro}
If a strict map $\{f^n:A^n\to B^n\}$ is a monomorphism in $\ProAb$,
and $\{A^n\}$ is a constant pro-abelian group, then $f^n$ is an
injection for all $n \gg 0$.
\end{lem}

\begin{proof}
Let $C^n = \ker(f^n)$. As noted above,  $\{C^n\}$ is the kernel
of $\{f^n\}$ in $\ProAb$ and hence it is the zero object of this
category. By Lemma \ref{lem:prozero}, this means
that for all $m$, there exists an $n \geq m$ such that $C^n
\to C^m$ is the zero map. But $\{A^n\}$ is constant and $C^n \subset
A^n$ for all $n$  and hence each map $C^n \to C^m$ is injective. It
follows that $C^n = 0$ for all $n \gg 0.$
\end{proof}

A {\em pro-spectrum} is contravariant functor from $\N$
to the category of spectra --- i.e., it is  a sequence of maps of spectra
of the form
$$
\cdots \to E^2 \to E^1.
$$
A strict map of pro-spectra is a natural transformation of such functors.
%For our purposes,
%a map of pro-spectra is a natural transformation
%of such functors. (That is, a map of pro-spectra will be a strict map.)
We say that a strict map of pro-spectra $\{E^n\} \to \{F^n\}$  is a
{\em weak equivalence} if for each $q\in\Z$ the induced (strict) map of
pro-abelian groups $\{\pi_q E^n\} \to \{\pi_q F^n\}$
%(which is a strict map, regarded as a morphism in $\ProAb$)
is an isomorphism in the category $\ProAb$.
A commutative square of pro-spectra and strict maps,
$$
\xymatrix{
\{E^n\} \ar[r] \ar[d] & \{F^n\} \ar[d] \\
\{G^n\} \ar[r] & \{H^n\},
}
$$
is said to be {\em homotopy cartesian} if the induced map of pro-spectra
$$
\{E^n\} \to \{ \holim (G^n \to H^n \leftarrow F^n) \}
$$
is a weak equivalence.

We arrive at the main theorem of this paper:

\begin{thm} \label{MainTheorem}
Let $\fc = \{c_1, c_2, \cdots\}$ be a sequence of integers with
$c_i\ge2$ for all $i\ge1$, and let $k$ be a regular ring containing
a field. Then for any $X$ in $\cMpctf$, the map
$$
\cK(X_k)^\fc \to \cKH(X_k)^\fc
$$
is an equivalence.
\end{thm}

\begin{proof}
If $\chr(k) = 0$, this was proved in Theorem \ref{dilation-0}.
Assume $\chr(k)=p>0$. The usual transfer argument,
involving prime-to-$p$ field extensions,
allows us to assume that $k$ contains an infinite field.

If $\cF$ is a functor defined on the category of $k$-schemes of
finite type, we will interpret $\cF$ in this proof as a functor on
$\cMpctf$ by precomposing with the $k$-realization functor.
If $\Box$ is any square in $\cMpctf$, write $\cF(\Box)$ for the
iterated homotopy fiber of the commutative square
of spectra obtained by applying $\cF$ to the diagram $\Box$.

With this notation, we claim $\cK(-)^\fc$ satisfies $cdh$ descent
on $\cMpctf$. To show this, we need to prove that
$\pi_q \cK(\Box)^\fc = 0$ for all $q \in \Z$, where
$\Box$ is either a Zariski square or an abstract blow-up square.

Fix an integer $q$. By \cite[Proposition 14.7]{chww-monoid},
we have an isomorphism in $\ProAb$:
$$
\{\pi_q \cK(\Box)\} \to \{\pi_q TC^n(\Box;p)\}.
$$
Lemma \ref{lem:constantpro} shows that
$\pi_q \cK(\Box) \to \pi_q TC^n(\Box;p)$ is injective for all
$n\ge n(q)$, where $n(q)$ depends on $q$.
Thus $\pi_q \cK(\Box)^\fc \to \pi_q TC^n(\Box;p)^\fc$
is also injective for all $n\ge n(q)$.
But for any $n$, $X\mapsto TC^n(X_k;p)^\fc$ satisfies $cdh$ descent
by Corollary \ref{Cor624}, and hence $\pi_q TC^n(\Box; p)^\fc = 0$.
Therefore $\pi_q \cK(\Box)^\fc = 0$, i.e.,
$\cK(-)^\fc$ satisfies $cdh$ descent.

By \cite[14.5]{chww-monoid}, $\cKH$ satisfies $cdh$ descent on $\cMpctf$
and hence so does $\cKH^\fc$, by Lemma \ref{colim-descent}.
It follows that the homotopy fiber $\cG$ of
$\cK(-)^\fc \to \cKH(-)^\fc$ also satisfies $cdh$ descent on
$\cMpctf$. Suppose that $k$ contains a field $F$.
For every smooth $Y$ in $\cMpctf$, $Y_F$ is smooth over $F$ and hence
$Y_k$ is smooth over $k$
(see \cite[6.5]{chww-monoid});
this implies that $\cK(Y_k)\to\cKH(Y_k)$ is a weak equivalence,
and hence that $\cG(Y)\sim *$.
By Proposition \ref{prop:cdhdescent},
$\cG(X) \sim *$ and hence $\cK(X_k)^\fc \sim \cKH(X_k)^\fc$
for all $X \in \cMpctf$.
\end{proof}

The special case $X=\MSpec(A)$ affirms
Gubeladze's dilation conjecture \cite[1.1]{Gu05} for $k[A]$:

\begin{cor}\label{Gconjecture}
Let $\fc = \{c_1, c_2, \cdots\}$ be a sequence of integers $\ge2$
and $k$ a regular ring containing a field. If $A$ is any
cancellative, torsionfree monoid with no non-trivial units, then
$$
\cK(k[A])^\fc \sim \cK(k).
$$
\end{cor}

\begin{proof}
Since $A$ is the direct limit of its finitely generated submonoids,
we may assume that $A$ is finitely generated, cancellative and torsionfree,
with no non-trivial units. The Separation Theorem \cite[p.\,13]{Fulton}
gives a group homomorphism $p:A^+\to\Z$ with $p(A\setminus\{0\})>0$ and
$p(-A)\le0$; using $p$, the ring $k[A]$ admits a grading by the
natural numbers with $k[A]_0 = k$. It follows that
$KH_*(k[A]) \cong K_*(k)$ via the canonical map and that the
action of $\theta_c$ on $KH_*(k[A])$ is trivial.
Taking $X=\MSpec(A)$, the assertion follows from Theorem \ref{MainTheorem}.
\end{proof}

\bibliographystyle{plain}

\end{document}